\documentclass[12pt]{amsart}

\usepackage{amsmath,bm}
\usepackage[color=yellow,textwidth=2cm,colorinlistoftodos]{todonotes}
\usepackage{hyperref}
\usepackage[nameinlink,capitalize]{cleveref}
\usepackage{dsfont}
\usepackage[english]{babel}

\newcommand{\summe}[2]{\displaystyle\sum_{#1}^{#2}}
\newcommand{\summezwei}[2]{\sum_{#1}^{#2}}

\newcommand{\integral}[4]{\displaystyle\int\limits_{#1}^{#2}{#3}\,{#4}}

\newcommand{\indikator}[2]{\mathds{1}_{{#1}}({#2})}
\newcommand{\indikatorzwei}[1]{\mathds{1}_{{#1}}}

\newcommand{\erwartung}[1]{\mathbb{E}({#1})}
\newcommand{\bigerwartung}[1]{\mathbb{E}\left({#1}\right)}

\newcommand{\stokonv}[1]{\xrightarrow[({#1}\rightarrow\infty)]{ \mathbb{P}  }}

\newcommand{\limes}[2]{\lim \limits_{#1 \rightarrow #2}}

\newcommand{\ft}[1]{\widehat{{#1}}}
\newcommand{\B}{\mathcal{B}}
\newcommand{\im}{\mathrm{i}}
\newcommand{\C}{\mathbb{C}}

\newcommand{\R}{\mathbb{R}}
\newcommand{\Q}{\mathbb{Q}}

\newcommand{\N}{\mathbb{N}}

\newcommand{\Pro}{\mathbb{P} } 
\newcommand{\eps}{\varepsilon} 
\newcommand{\secret}[1]{}

\newcommand{\norm}[1]{\| {#1}\|}
\newcommand{\Li}[1]{\text{L}(\R^{#1})}

\newcommand{\bignorm}[1]{ \left \Vert {#1} \right \Vert}
\newcommand{\skp}[2]{ \langle {#1},{#2} \rangle}

\newcommand{\skpzwei}[2]{ \left \langle {#1},{#2} \right \rangle}

\newtheorem{theorem}{Theorem}[section]
\newtheorem{lemma}[theorem]{Lemma}
\newtheorem{prop}[theorem]{Proposition}
\newtheorem{cor}[theorem]{Corollary}

\theoremstyle{definition}
\newtheorem{defi}[theorem]{Definition}
\newtheorem{example}[theorem]{Example}

\theoremstyle{remark}
\newtheorem{remark}[theorem]{Remark}

\numberwithin{equation}{section}
\begin{document}
\sloppy
\title[]{Multivariate stochastic integrals with respect to independently scattered random measures on $\delta$-rings}

\author{D. Kremer}
\address{Dustin Kremer, Department Mathematik, Universit\"at Siegen, 57068 Siegen, Germany}
\email{kremer\@@{}mathematik.uni-siegen.de}

\author{H.-P. Scheffler}
\address{Hans-Peter Scheffler, Department Mathematik, Universit\"at Siegen, 57068 Siegen, Germany}
\email{scheffler\@@{}mathematik.uni-siegen.de}

\date{\today}


%Abstract
\begin{abstract}
In this paper we construct general vector-valued infinitely-divisible independently scattered random measures with values in $\R^m$ and their corresponding stochastic integrals. Moreover, given such a random measure, the class of all integrable matrix-valued deterministic functions is characterized in terms of certain characteristics of the random measure. In addition a general construction principle is presented.

\end{abstract}
\maketitle
%% Section 1: Introduction
\section{Introduction}
Various stochastic processes and random fields are built by integrating a family of deterministic functions with respect to an infinitely-divisible random measure (e.g. a noise). One of the first and most prominent examples is the fractional Brownian motion. This was extended to the so called fractional stable motion by replacing the Gaussian random measure by a symmetric $\alpha$-stable ($S \alpha S$) random measure, see \cite{SaTaq94} for details. \\
Based on $S \alpha S$ random measures a vast class of stochastic processes and random fields has been constructed. See e.g. \cite{multi}, \cite{bms}, \cite{hoffmann}, \cite{SaTaq94} and  \cite{Sato} to name a few. All these processes and fields are univariate and have $S \alpha S$ marginal distributions by construction. The general theory of arbitrary infinitely-divisible independently scattered random measures (ISRMs) and the class of integrable functions was carried out in \cite{RajRo89}.\\
Surprisingly enough much less is known in the multivariate case. Besides the Gaussian case and an ad hoc construction of a multivariate $S \alpha S$ random measure in \cite{lixiao}, there appears to be no general theory of multivariate random measures.  The purpose of this paper is to carefully develop an honest theory of general infinitely-divisible ISRMs and their corresponding integrals for matrix-valued deterministic functions. Our approach follows along the lines of \cite{RajRo89}. However, since we construct vector-valued measures, some univariate methods using monotonocity no longer apply. \\
 One may argue that infinite divisibility is a rather strong property. However, we show that an atomless random measure (see \cite{Pre57}) is necessarily infinitely-divisible (i.d.), so i.d. is quite a natural assumption. In a subsequent paper \cite{paper2} our methods will be used to construct an $\R^m$-valued ISRM with operator-stable marginals. \\
The paper is organized as follows. We start with some notation and useful preliminaries about infinitely-divisible distributions and $\delta$-rings in section 2. We then characterize all infinitely-divisible $\R^m$-valued random measures in section 3, already suggesting a complex-valued point of view and proposing a useful construction principle in \cref{08031703}. Section 4 is devoted to an insertion about atomless random measures and its connection to infinite divisibility. Finally, in section 5 the integrators provided by section 3 are used to define the corresponding stochastic integral for matrix-valued functions. Here we will characterize the class of integrable functions (w.r.t. to a given random measure) and clarify the intimate relation between the real-valued and complex-valued perspective as announced before.
%% Section 2: Preliminaries
\section{Preliminaries}
Let L$(\mathbb{K}^m)$ denote the set of all linear operators on $\mathbb{K}^m$, represented as $m \times m$ matrices with entries from $\mathbb{K}$, where $\mathbb{K}$ is either $\R$ or $\C$. Furthermore let $\norm{\cdot}$ be the Euclidian norm on $\R^m$ with inner product $\skp{\cdot}{\cdot}$ while the identity operator on $\R^m$ is denoted by $I_m$. Then it is well-known (\textit{L\'{e}vy-Khintchine-Formula}, see \cite{thebook}) that $\varphi=\exp(\psi)$ with $\psi:\R^m \rightarrow \C$ is the Fourier transform (or characteristic function) of an infinitely-divisible distribution on $\R^m$, if and only if $\psi$ can be represented as
\begin{equation*}
\psi(t)=\im  \skp{\gamma}{t} - \frac{1}{2} \skp{Qt}{t} + \integral{\R^m}{}{\left( \text{e}^{\im \skp{t}{x}} - 1 - \frac{\im \skp{t}{x} }{1+\norm{x}^2} \right)}{\phi(dx)} , \quad t \in \R^m
\end{equation*}
for a \textit{shift} $\gamma \in \R^m$, some \textit{normal component} $Q \in \Li{m}$ which is symmetric and positive semi-definite and a L\'{evy} measure $\phi$, i.e. $\phi$ is a measure on $\R^m$ with $\phi(\{0\})=0$ and $\int_{\R^m} \min\{1,\norm{x}^2\} \, \phi(dx)< \infty$. For the distribution $\mu$ with $\ft{\mu}=\varphi$ we write $\mu \sim [\gamma, Q, \phi]$ as $\gamma, Q$ and $\phi$ are uniquely determined by $\mu$. $\psi$ is the only continuous function with $\psi(0)=0$ and $\ft{\mu}=\exp(\psi)$, subsequently referred to as the \textit{log-characteristic function} of $\mu$. 
\begin{lemma} \label{23111604}
Let $(\mu_n)$ be a sequence of i.d. distributions on $\R^m$. Then $\mu_n \sim [\gamma_n,Q_n,\phi_n]$ converges weakly to the point measure in zero $\eps_0$ if and only if $\gamma_n \rightarrow 0, Q_n \rightarrow 0$ and
\begin{equation} \label{eq:21111601}
 \integral{\R^m}{}{\min \{1,\norm{x}^2\} }{\phi_n(dx)} \rightarrow 0 \quad (n \rightarrow \infty).
\end{equation}
\end{lemma}
\begin{proof}
By Theorem 3.1.16 in \cite{thebook} it obviously remains to check that \eqref{eq:21111601} is equivalent to $\phi_n(A) \rightarrow 0$ for all Borel sets $A$ which are bounded away from zero together with 
\begin{equation*}
\limes{\eps}{0} \, \limes{n}{\infty} \integral{\{x: 0< \norm{x} < \eps\} }{}{\skp{t}{x}^2}{\phi_n(dx)}= 0 \quad \text{for all $t \in \R^m$.}
\end{equation*} 
Therefore, by distinguishing the sign of each component, we can decompose $\R^m$ into sets $M_j$ $(j=1,...,2m)$ such that $\norm{x}^2\le \norm{x}_1^2=\skp{t_j}{x}^2$ for all $x \in M_j$ and suitable $t_j \in \{-1,1\}^m$.
\end{proof}
Throughout this paper let $S$ be any non-empty set. Then a family of sets $\mathcal{S} \subset \mathcal{P}(S):=\{A: A \subset S \}$ is called a \textit{$\delta$-ring} (on $S$), if it is a ring (i.e. closed under union and difference together with $\emptyset \in \mathcal{S}$) such that there is a sequence $(S_n) \subset \mathcal{S}$ with $\cup_{n=1}^{\infty} S_n=S$ and which is additionally closed unter countably many intersections. Using the properties of a ring, the sequence $(S_n)$ can assumed to be increasing as well as disjoint, depending on the respective occurrence. Furthermore, write
\begin{equation} \label{eq:21161101}
\bigcup_{n=1}^{\infty} A_n =A \setminus \bigcap_{n=1}^{\infty}(A \setminus A_n) \in \mathcal{S} 
\end{equation}
for $A \in \mathcal{S}$ and any sequence $(A_n)\subset \mathcal{S}$ with $A_n \subset A$, to observe that $\delta$-rings behave locally like $\sigma$-algebras. Particularly any $\delta$-ring $\mathcal{S}$ with $S \in \mathcal{S}$ is a $\sigma$-algebra. The next result is also elementary, but helpful, where $\sigma(\mathcal{S})$ denotes the $\sigma$-algebra on $S$ that is generated by $\mathcal{S}$.
\begin{lemma} \label{15031702}
Let $\mathcal{S}$ be a $\delta$-ring on $S$. Then $A \cap M \in \mathcal{S}$ for all $A \in \mathcal{S}$ and $M \in \sigma(\mathcal{S})$.
\end{lemma}
\begin{proof}
Obviously we have $\mathcal{S} \subset \mathcal{D}:=\{M \in \sigma(\mathcal{S})  \, | \, \forall A \in \mathcal{S} :A \cap M \in \mathcal{S}\}$. Then we just have to check that $\mathcal{D}$ is already a $\sigma$-algebra on $S$. Because of $A \cap M^c = A \setminus (A \setminus M)$ it follows that $M^c \in \mathcal{D}$, whenever $M \in \mathcal{D}$ is true. Analogously we see that $\mathcal{D}$ is closed under countably many unions as $A \cap (\cup_{n=1}^{\infty} M_n)^c=\cap_{n=1}^{\infty} A \setminus (A \cap M_n) \in \mathcal{S}$ for $A \in \mathcal{S}$ arbitrary and any sequence $(M_n) \subset \mathcal{D}$. 
\end{proof}
We now want to consider vector-valued set functions with domain $\mathcal{S}$. For our purpose it is sufficient to assume that $V$ is a Banach space (with norm $\norm{\cdot}_V$). Then we call a set function $T:\mathcal{S} \rightarrow V$ \textit{additive}, if $T(\emptyset)=0$ and $T(A_1 \cup \cdots \cup A_k)=T(A_1)+ \cdots + T(A_k)$ for 
any $k \in \N$ and disjoint sets $A_1,...,A_k \in \mathcal{S}$. Furthermore, if 
\begin{equation} \label{eq:22111601}
T \left( \cup_{n=1}^{\infty} A_n \right) = \summe{n=1}{\infty} T(A_n) \quad \text{w.r.t. $\norm{\cdot}_V$}
\end{equation}
holds for any disjoint sequence $(A_n) \subset \mathcal{S}$ with $\cup_{n=1}^{\infty} A_n \in \mathcal{S}$, then $T$ is called \textit{$\sigma$-additive}. Finally $\sigma$-additive set functions on $\sigma$-algebras are called \textit{vector measures}. As we claim $T(A) \in V$ for every $A \in \mathcal{S}$ one can use standard arguments (see 1.36 in \cite{Kle08} for example) to show that an additive set function $T: \mathcal{S} \rightarrow V$ is $\sigma$-additive if and only if
\begin{equation} \label{eq:22111602}
\text{$V$-}\limes{n}{\infty} \, T(A_n) =0 \quad \text{for all $(A_n)\subset \mathcal{S}$ with $A_n \downarrow \emptyset$.}
\end{equation}
In this context we distinguish the previous definition from the term \textit{pre-measure}, i.e. those $\sigma$-additive set functions on $\mathcal{S}$ that take values in $[0, \infty]$. Yet, given any set function $T:\mathcal{S} \rightarrow V$, the \textit{total variation} $|T|$ (of $T$) connects these concepts:
\begin{equation*}
|T|(A):=\sup \left \{ \sum_{j=1}^{n} \norm{T(A_j)}_V  \, \, | \, \, \text{$n \in \N$ and $A_1,...,A_n \in \mathcal{S}$ disjoint with $A_j \subset A$} \right \}, \quad A \in \mathcal{S}.
\end{equation*}
\begin{theorem} \label{09121602}
Let $T: \mathcal{S} \rightarrow V$ be a $\sigma$-additive set function. Then $|T|$ is a pre-measure. Additionally, if $V$ is finite-dimensional, then $|T|$ is $[0,\infty)$-valued, i.e. a finite pre-measure.
\end{theorem}
\begin{proof}
As in III 1, Lemma 6 in \cite{DunSch57} we get that $|T|$ is additive, although $\mathcal{S}$ is just a ($\delta$-)ring. Using this and the arguments in the proof of III 4, Lemma 7 in \cite{DunSch57} it follows that $|T|$ is even $\sigma$-additive. Finally, if $V=\R^n$ (without loss of generality), we can assume that $n=1$ by equivalence of norms and by considering the component functions of $T$ which inherit the $\sigma$-additivity. Now, due to \eqref{eq:22111602} and the closure of $\mathcal{S}$ under countably many intersections, we can argue as in XI, Theorem 8 in \cite{La68} to obtain the assertion.
\end{proof}
\begin{remark} \label{24041701}
In view of the quoted proofs we observe that the previous statement remains true for any \textit{$\sigma$-subadditive} set function on $\mathcal{S}$ which is $[0,\infty)$-valued. 
\end{remark}
Unfortunately, it is impossible to formulate the \textit{Hahn-Jordan-decomposition} on $\delta$-rings. But for the case $V=\R$ we can at least consider the \textit{positive variation} $T^{+}:\mathcal{S} \rightarrow [0,\infty)$ and the \textit{negative variation} $T^{-}:\mathcal{S} \rightarrow [0,\infty)$ of the $\sigma$-additive set function $T$, defined by $T^{\pm}(A):= \frac{1}{2}(|T|(A) \pm T(A))$, respectively. Then it is clear that $T^{+}$ and $T^{-}$ are finite pre-measures with $T=T^{+}-T^{-}$ and $|T|=T^{+}+T^{-}.$ Although it was formulated for $\sigma$-algebras in \cite{DunSch57} (see III 1, Theorem 8), we immediately see that the following representations hold for every $A \in \mathcal{S}$:
\begin{equation} \label{eq:23111601}
T^{+}(A)=\sup \{T(B): B \in \mathcal{S} \text{ with } B \subset A\}
\end{equation}
and
\begin{equation} \label{eq:23111602}
T^{-}(A)=- \inf \{T(B): B \in \mathcal{S} \text{ with } B \subset A\}.
\end{equation}
%% Section 3: ISRM
\section{Infinitely-divisible random measures}
In this section we define and analyze ISRMs with values in $\mathbb{K}^m$ defined on $\delta$-Rings. Hence if we denote by $L^{0}(\Omega, \mathbb{K}^m)$ the set of all $\mathbb{K}^m$-valued random vectors defined on any abstract probability space $(\Omega, \mathcal{A}, \Pro)$, a mapping $M:\mathcal{S} \rightarrow L^{0}(\Omega,\mathbb{K}^m)$ is shortly called an independently scattered random measure (on $\mathcal{S}$ with values in $\mathbb{K}^m$), if the following conditions hold:
\begin{itemize}
\item[($RM_1$)] For every finite choice $A_1,...,A_k$ of disjoint sets in $\mathcal{S}$ the random vectors $M(A_1),...,M(A_k)$ are stochastically independent.
\item[($RM_2$)] For every sequence $(A_n) \subset \mathcal{S}$ of disjoint sets with $\cup_{n=1}^{\infty} A_n \in \mathcal{S}$ we have
\begin{equation*}
M (\cup_{n=1}^{\infty} A_n)=\summe{n=1}{\infty} M(A_n) \quad \text{almost surely (a.s.).}
\end{equation*}
\end{itemize}
By introducing the mapping $\Xi_{(m)}(z):= (\text{Re } z, \text{Im }z) \in \R^{2m}$ for $z \in \C^m$, condition $(RM_1)$ here means independence of $\Xi(M(A_1)),..., \Xi(M(A_k))$. Furthermore and with an analogous extension for $\mathbb{K}=\C$ we call such an ISRM \textit{infinitely-divisible}, if this true for (the distribution of) every random vector $M(A), A \in \mathcal{S}$. It will turn out later that it is quite natural to concentrate on infinitely-divisible random measures. In this case we get the following characterization where we first consider $\mathbb{K}=\R$:
\begin{theorem} \label{09121603}
Let $M$ be an i.d. ISRM on $\mathcal{S}$ with values in $\R^m$, where $M(A) \sim [\gamma_A, Q_A, \phi_A]$ for every $A \in \mathcal{S}$. Then we have:
\begin{itemize}
\item[(a)] The mapping $\mathcal{S} \ni A \mapsto \gamma_A \in \R^m$ is $\sigma$-additive.
\item[(b)] The mapping $\mathcal{S} \ni A \mapsto Q_A \in \Li{m}$ is $\sigma$-additive.
\item[(c)] The mapping $\mathcal{S} \ni A \mapsto \phi_A(B)$ is a finite pre-measure for every fixed Borel set $B$ which is bounded away from zero.
\end{itemize}
Conversely, for every family of triplets $ ([\gamma_A, Q_A, \phi_A])_{A \in \mathcal{S}}$ that satisfies (a)-(c) there exists an i.d. ISRM $M$ (on some suitable probability space) with $M(A) \sim [\gamma_A, Q_A, \phi_A]$ for every $A \in \mathcal{S}$. Furthermore, the finite-dimensional distributions of $M$ are uniquely determined by the latter property.
\end{theorem}
\begin{proof}
Assume first that $M$ is an infinitely-divisible ISRM. Since $M(\emptyset)=0$ a.s., the additivity of the mappings in (a)-(c) can be easily deduced from the L\'{e}vy-Khintchine-Formula and its uniqueness statement by using $(RM_1)$ and $(RM_2)$ for only finitely many sets. Then it is even clear that $\phi_{A_1 \cup \cdots \cup A_k}$ equals the measure $\phi_{A_1}+ \cdots +\phi_{A_k}$. Now let $(B_n) \subset \mathcal{S}$ be a sequence with $(B_n) \downarrow \emptyset$ and define the auxiliary sequence $C_1=\emptyset, C_n= B_{n-1} \setminus B_n$ (for $n \ge 2$) to observe that
\begin{equation*}
M(B_1)= M(\cup_{k=1}^{\infty} C_n )=\limes{k}{\infty} (M(B_1)-M(B_k)),
\end{equation*}
which leads to $M(B_k) \rightarrow 0$ a.s. Then (a) and (b) follow by \cref{23111604} together with \eqref{eq:22111602}. Similarly using Theorem 3.1.16 in \cite{thebook} we obtain (c).\\
Concerning the second part denote by $\Theta(A,\cdot)$ the log-characteristic function of the i.d. distribution on $\R^m$ with triplet $[\gamma_A,Q_A,\phi_A]$ for $A \in \mathcal{S}$. Moreover, for any $n \in \N$ and $A_1,...,A_n \in \mathcal{S}$ we define
\begin{equation*} 
\psi_{A_1,...,A_n}(t):= \summe{J \subset \{1,...,n\} }{} \Theta \left (\mathcal{Z}_J^{(n)}, \sum \nolimits_{j \in J} t_j \right),
\end{equation*}
where $t=(t_1,...,t_n) \in \R^{n \cdot m}$ and
\begin{equation*}
\mathcal{Z}_J^{(n)}:=\mathcal{Z}_{J}(A_1,...,A_n): =\begin{cases} \emptyset,  &\text{if $J = \emptyset$} \\ \left[ \bigcap_{j \in J}^{} A_j \setminus \bigcup_{l \in J^c}^{} A_l \right], &\text{if $J \ne \emptyset$}    \end{cases} \quad \in \mathcal{S}.
\end{equation*}
Then, with Lemma 3.5.9 in \cite{Jac01} for example, it easy to see that $\exp(\psi_{A_1,...,A_n}(\cdot))$ is not only continuous, but also positive semi-definite in the sense of Bochner's theorem as this is true for the functions $\exp(\Theta(A,\cdot))$ already. Then by the theorem itself we obtain the existence of a distribution $\mu_{A_1,...,A_n}$ on $\R^{n \cdot m}$ whose Fourier transform is given by $\exp(\psi_{A_1,...,A_n}(\cdot))$, in particular we have $\mu_A \sim [\gamma_A,Q_A, \phi_A]$ for all $A \in \mathcal{S}$.  Then on one hand we can check that
\begin{equation*}
\mathcal{Z}_{J}(A_1,...,A_{n+1}) \cup \mathcal{Z}_{J\cup\{n+1\}}(A_1,...,A_{n+1}) =\mathcal{Z}_{J}(A_1,...,A_{n}) 
\end{equation*}
for $A_1,...,A_{n+1} \in \mathcal{S}$ and every $J \in \mathcal{P}(\{1,...,n\}) \setminus \emptyset$, where the union is disjoint. On the other hand we can use (c) again to show that $\Theta(B_1 \cup B_2,t)=\Theta(B_1,t)+\Theta(B_2,t)$ for all $B_1,B_2 \in \mathcal{S}$ disjoint and $t \in \R^m$. Hence for $t_1,...,t_n \in \R^m$ we get with $t_{n+1}:=0$ that
\begin{align*}
\psi_{A_1,...,A_{n+1}}(t_1,...,t_n,0) &= \summe{J \subset \{1,...,n+1\} }{}  \Theta(\mathcal{Z}_J^{(n+1)}, \sum_{j \in J} t_j) \\
&= \summe{J \subset \{1,...,n\} }{}  \left[ \Theta(\mathcal{Z}_J^{(n+1)}, \sum_{j \in J} t_j) + \Theta(\mathcal{Z}_{J\cup \{n+1\}}^{(n+1)}, \sum_{j \in J} t_j) \right] \\
&= \summe{\substack{ J \subset \{1,...,n\}, \\ J \ne \emptyset }}{}  \left[ \Theta(\mathcal{Z}_J^{(n+1)}, \sum_{j \in J} t_j) + \Theta(\mathcal{Z}_{J\cup \{n+1\}}^{(n+1)}, \sum_{j \in J} t_j) \right] \\
&= \psi_{A_1,...,A_{n}}(t_1,...,t_n).
\end{align*}
Overall this mostly proves that the considered system is projective and by Kolomogorov's consistency theorem there exists a probability space $(\Omega, \mathfrak{A}, \Pro)$ and a family $M=\{M(A):A \in \mathcal{S}\}$ of random vectors with $\mathcal{L}((M(A_1),...,M(A_n)))=\mu_{A_1,...,A_n}$. For $A_1,...,A_n \in \mathcal{S}$ disjoint we have that $\mathcal{Z}_J^{(n)}=A_j$, if $J=\{j\}$ and $\mathcal{Z}_J^{(n)}=\emptyset$ else, which yields that $(RM_1)$ is fulfilled. For $(RM_2)$ we first fix $A_1,A_2 \in \mathcal{S}$ arbitrary and write 
\begin{center}
$\ft{\mathcal{L}}(M(A_1 \cup A_2)-M(A_1)-M(A_2))(t)= \ft{\mu}_{A_1 \cup A_2,A_1,A_2}(t,-t,-t), \quad t \in \R^m$ 
\end{center}
to see that $M$ is finitely additive as the right-hand side equals $1$ by construction. Thus for a sequence like given in $(RM_2)$ it suffices to show that
\begin{equation*}
M (\cup_{j=1}^{\infty} A_j) - M (\cup_{j=1}^{k} A_j)= M (\cup_{j=k+1}^{\infty} A_j) \stokonv{k} 0
\end{equation*}
by a straight-forward multivariate extension of the the three-series-theorem (see Theorem 9.7.1 in \cite{Dud02}) and by what we have shown before. If we let $B_k:=\cup_{j=k+1}^{\infty} A_j$ with $B_k \downarrow \emptyset$, it follows by (a) and (b) that $\gamma_{B_k} \rightarrow 0$ as well as that $Q_{B_k} \rightarrow 0$. Provided that
\begin{equation} \label{eq:01121601}
\limes{k}{\infty} \integral{\R^m}{}{\min \{1,\norm{x}^2\}}{\phi_{B_k}(dx)} =0,
\end{equation}
the assertion would follow via \eqref{eq:22111602}. Fix $\eps>0$ and choose $\delta>0$ sufficiently small such that 
\begin{equation*}
\integral{ \{x: \norm{x}< \delta \}  }{}{\min \{1,\norm{x}^2\}}{\phi_{B_k}(dx)} \le \integral{\{x: \norm{x}< \delta \}}{}{\min \{1,\norm{x}^2\}}{\phi_{B_1}(dx)} < \eps, \quad k \in \N
\end{equation*}
in face of $\phi_{k+1} \le \phi_{k}$ (see above), such that \eqref{eq:01121601} follows by (c) again. Finally for uniqueness we merely consider $A_1,A_2 \in \mathcal{S}$ and write
\begin{equation*}
\skp{t_1}{M(A_1)}+\skp{t_2}{M(A_2)}=\skp{t_1}{M(A_1 \setminus A_2)} + \skp{t_1+t_2}{M(A_1 \cap A_2)} + \skp{t_2}{M(A_2 \setminus A_1)}
\end{equation*}
for $t_1,t_2 \in \R^m$ and by $(RM_2)$, where the random variables on the ride side are independent due to $(RM_1)$. Now the statement can be deduced easily.
\end{proof}
Let us remark that the previous theorem as well as the following ones are similar to the corresponding, but univariate results in \cite{RajRo89}.
\begin{theorem} \label{20031705}
Let $M$ be an i.d. ISRM as before, then there exists a $\sigma$-finite measure $\lambda_M$ on $\sigma(\mathcal{S})$, called \textit{control measure of $M$}, which is uniquely determined by
\begin{equation} \label{eq:09121601}
\lambda_M(A)=|\gamma|_A + \text{tr}(Q_{A}) + \integral{\R^m }{}{\min \{1,\norm{x}^2\}}{\phi_{A}(dx)}, \quad A \in \mathcal{S}.
\end{equation}
Furthermore, for any sequence $(A_n) \subset \mathcal{S}$ we have:
\begin{itemize}
\item[(i)] $\lambda_M(A_n) \rightarrow 0$ implies $M(A_n) \rightarrow 0$ in probability.
\item[(ii)] If $M(A_n') \rightarrow 0$ in probability for every sequence $(A_n') \subset \mathcal{S}$ with $A_n' \subset A_n$, then it follows that $\lambda_M(A_n) \rightarrow 0$.
\end{itemize}
\end{theorem}
\begin{proof}
We have to show that \eqref{eq:09121601} defines a finite pre-measure on $\mathcal{S}$, then $\lambda_M$ would be its unique extension on $\sigma(\mathcal{S})$: Non-negativity is obvious. Morevover $|\gamma|$ is finite by \cref{09121602} and \cref{09121603} (a). The mapping $A \mapsto \text{tr}(Q_A)$ preserves the $\sigma$-additivity in \cref{09121603} (b) by continuity of the trace-mapping $\text{tr}(\cdot)$. Finally we could already show that $A \mapsto \phi_A$ is additive, thus as before it remains to show that
\begin{equation} \label{eq:09121604}
\integral{\R^m}{}{\min \{1,\norm{x}^2\}}{\phi_{B_n}(dx)} \rightarrow 0 
\end{equation} 
for any sequence $(B_n) \subset \mathcal{S}$ with $B_n \downarrow \emptyset$. Actually, the previous proof even revealed that $M(B_n) \rightarrow 0$ a.s., such that \eqref{eq:09121604} follows by \eqref{eq:21111601}. \\
Now, if $\lambda_M(A_n) \rightarrow 0$ for a sequence as above, the same holds for each of the corresponding expressions in \eqref{eq:09121601} which allows us to use \cref{23111604} again. Because of $\norm{\gamma_{A_n}} \le |\gamma|_{A_n}$ and since $\text{tr}(Q_{A_n}) \rightarrow 0$ implies $Q_{A_n} \rightarrow 0$ we get $M(A_n) \rightarrow 0$ in probability. Conversely, the proof of $\lambda_M(A_n) \rightarrow 0$ reduces to the verification of $|\gamma|_{A_n} \rightarrow 0$ after using similar arguments as before and especially the assumption that $M(A_n) \rightarrow 0$ in probability. Consider the component functions $\gamma^{(1)},...,\gamma^{(m)}$ and fix some $\eps>0$ and $j \in \{1,...,m\}$, where \cref{09121603} (a) and the combination of \eqref{eq:23111601}-\eqref{eq:23111602} guarantee the existence of sequences $(A_{n,i})_n \subset \mathcal{S}$ with $A_{n,i} \subset A_n$ for $i=1,2$ with
\begin{equation*}
|\gamma^{(j)}|_{A_n} \le \gamma^{(j)}_{A_{n,1}} - \gamma^{(j)}_{A_{n,2}}  + \eps, \quad n \in \N.
\end{equation*}
Now one can use the given assumption together with \cref{23111604} again to see that $\gamma_{A_{n,i}}^{(j)} \rightarrow 0$ for $i=1,2$ which yields $|\gamma^{(j)}|_{A_n} \rightarrow 0$ and therefore the assertion of (ii), see the proof of \cref{09121602}.
\end{proof}
Next we want to extend Lemma 2.3 in \cite{RajRo89} which yields a construction principle for ISRMs in \cref{08031703} (b) below : Given measurable spaces $(\Omega_1,\mathcal{A}_1)$ and $(\Omega_2,\mathcal{A}_2)$, a mapping $\kappa: \Omega_1 \times \mathcal{A}_2 \rightarrow [0,\infty]$ is called a \textit{simultaneous $\sigma$-finite transition function from $\Omega_1$ to $\Omega_2$}, if the following conditions hold:
\begin{itemize}
\item[(i)] $\omega_1 \mapsto \kappa(\omega_1,A_2)$ is $\mathcal{A}_1$-$\B([0,\infty])$-measurable for every $A_2 \in \mathcal{A}_2$.
\item[(ii)] $A_2 \mapsto \kappa(\omega_1,A_2)$ is a measure on $(\Omega_2,\mathcal{A}_2)$ for every $\omega_1 \in \Omega_1$. Moreover there exist sequences $(A_{2,n})\subset \mathcal{A}_2$ and $(r_n) \subset [0,\infty)$ such that
\begin{equation} \label{eq:09051701}
\bigcup_{n=1}^{\infty}A_{2,n}=\Omega_2 \quad \text{and} \quad \forall n \in \N \, \forall \omega_1 \in \Omega_1: \kappa(\omega_1,A_{2.n}) \le r_n.
\end{equation}
\end{itemize}
Furthermore, if $\kappa(\omega_1,\cdot)$ is a probability measure for every $\omega_1 \in \Omega$, we say that $\kappa$ is \textit{Markovian}.
\begin{prop}
Let $(\Omega_1,\mathcal{A}_1, \nu)$ be a $\sigma$-finite measure space and $\kappa$ a simultaneous $\sigma$-finite transition function from $\Omega_1$ to $\Omega_2$. Then there exists a unique $\sigma$-finite measure $\nu \odot \kappa$ on the product space $(\Omega_1 \times \Omega_2, \mathcal{A}_1 \otimes \mathcal{A}_2)$ with the property
\begin{equation} \label{eq:06031701}
(\nu \odot \kappa)(A_1 \times A_2) = \integral{A_1}{}{\kappa(\omega_1,A_2)}{\nu(d \omega_1)} \quad \text{for all $A_1 \in \mathcal{A}_1,A_2 \in \mathcal{A}_2$}.
\end{equation}
Moreover, we have
\begin{equation} \label{eq:06031702}
\integral{\Omega_1 \times \Omega_2}{}{f(x)}{(\nu \odot \kappa) (dx)}= \integral{\Omega_1}{}{ \integral{\Omega_2}{}{f(\omega_1,\omega_2)}{\kappa(\omega_1,d \omega_2)}  }{\nu(d \omega_1)}
\end{equation}
for every measurable $f:\Omega_1 \times \Omega_2 \rightarrow \R$ that is non-negative or integrable w.r.t. $\mu \odot \kappa$.
\end{prop}
\begin{proof}
Choose $(A_{1,n}) \subset \mathcal{A}_1$ disjoint with $\cup_{n=1}^{\infty} A_{1,n}=\Omega_1$ and $\nu(A_{1,n})< \infty$ for all $n \in \N$. Let $\nu^{(n)}(\cdot):= \nu(\cdot \cap A_{1,n})$. Similarly $\kappa^{(n)}(\omega_1,\cdot):=\kappa(\omega_1,\cdot \cap A_{2,n})$ is a finite transition function with $(A_{2,n})$ from \eqref{eq:09051701} for every $\omega_1 \in \Omega_1$ and $n \in \N$. As the assertion is well-known for $\nu$ and $\kappa$ being finite (see 14.23 and 14.29 in \cite{Kle08}), one easily checks that it is enough to define
\begin{equation*}
(\nu \odot \kappa)(C):= \integral{\Omega_1}{}{ \integral{\Omega_2}{}{\indikator{C}{\omega_1,\omega_2}     }{\kappa(\omega_1, d \omega_2) } }{ \nu (d \omega_1)}, \quad C \in \mathcal{A}_1 \otimes \mathcal{A}_2.
\end{equation*}
More precisely we can consider $C_n:= A_{1,\pi_1(n)} \times A_{2,\pi_2(n)}$ with a suitable mapping $\pi=(\pi_1,\pi_2):\N \rightarrow \N^2$ which is one-to-one. Then $(\nu \odot \kappa)(\cdot \cap C_n)$ is finite under the given assumption on $\kappa$ and moreover equals $\nu^{(\pi_1(n))} \odot \kappa^{(\pi_2(n))}$ for every $n \in \N$.  
\end{proof}
\begin{theorem} \label{08031703}
Let $\mathcal{S}$ be a $\delta$-ring as above and consider the $\sigma$-algebra $\sigma(\mathcal{S})$.
\begin{itemize}
\item[(i)] For every i.d. ISRM $M$ on $\mathcal{S}$ with values in $\R^m$ there exists a simultaneous $\sigma$-finite transition function $\rho_M$ from $S$ to $\R^m$ with $(\lambda_M \odot \rho_M)(A \times B)= \phi_A(B)$ for every $A \in S$ and $B \in \B(\R^m)$, where $\phi_A$ is the L\'{evy} measure of $M(A)$. Here $\rho_M$ is uniquely determined $\lambda_M$-almost everywhere (a.e.) and can be chosen such that
\begin{equation} \label{eq:07031701}
\integral{\R^m}{}{\min \{1,\norm{x}^2\} }{\rho_M(s,dx)} \le 1 \quad \text{for every $s \in S$.}
\end{equation}
\item[(ii)] Conversely, let $\lambda$ be a measure on $S$ which is finite on $\mathcal{S}$ and $\rho$ a transition function from $S$ to $\R^m$ fulfilling \eqref{eq:07031701}, i.e. being simultaneous $\sigma$-finite. Then there exists an ISRM $M$ with $\lambda=\lambda_M$ and $\rho=\rho_M$ (in the previous sense). 
\end{itemize}
\end{theorem}
\begin{proof}
Assume the sequence $(S_n)\subset \mathcal{S}$ to be disjoint for this proof. Then, as in the proof of \cref{20031705}, we see that $Q_0^{*}(A,B):= \int_B \min\{1,\norm{x}^2\} \, \phi_A(dx)$ is a finite pre-measure on $\mathcal{S}$ for any fixed Borel set $B\subset \R^m$ and we denote its unique extension towards a $\sigma$-finite measure on $\sigma(\mathcal{S})$ by $Q_0(\cdot,B)$. Hence for $A \in \sigma(\mathcal{S})$ and $(B_k)\subset \B(\R^m)$ disjoint we observe by \cref{15031702} that
\begin{equation*}
Q_0 \left (A, \bigcup_{k=1}^{\infty}B_k \right)=  \summe{n=1}{\infty} \summe{k=1}{\infty} Q_0^{*}(A \cap S_n,B_k)=  \summe{k=1}{\infty}  \summe{n=1}{\infty} Q_0^{*}(A \cap S_n,B_k) =   \summe{k=1}{\infty}  Q_0(A,B_k).
\end{equation*}
Consequently the assumptions of Proposition 2.4 in \cite{RajRo89} are fulfilled and by a slight refinement (in particular $(\R^m, \B(\R^m))$ and $(\R, \B(\R))$ are isomorphic as measurable spaces) we get the existence of a Markovian transition function $\kappa$ from $S$ to $\R^m$ such that $Q_0(A,B)=(\lambda_0 \odot \kappa)(A \times B)$ for every $A \in \sigma(\mathcal{S})$ and $B \in \B(\R^m)$, where $\lambda_0(\cdot):=Q_0(\cdot,\R^m) \le \lambda_M(\cdot)$. Let $\tau_0$ be a $\lambda_M$-derivative of $\lambda_0$ with $\tau_0(s) \le 1$ for every $s \in S$ and set  
\begin{equation*} 
\rho_M(s,dx):= \tau_0(s) \cdot  \min \{1,\norm{x}^2\}^{-1} \cdot \indikator{\R^m \setminus \{0\} }{x} \kappa(s,dx), \quad s \in S.
\end{equation*}
This shows \eqref{eq:07031701}. Hence the following calculation, which is valid for every $A \in \mathcal{S},B \in \B(\R^m)$ and benefits from the simplicity of the integrand, yields
\allowdisplaybreaks
\begin{align*}
\integral{A}{}{\rho_M(s,B)}{\lambda_M(ds)} &= \integral{A}{}{ \integral{B \setminus \{0\} }{}{(\min\{1,\norm{x}^2\} )^{-1} }{\kappa(s,dx)} }{\lambda_0(ds)} \\
&= \integral{A \times (B \setminus \{0\})}{}{(\min\{1,\norm{x}^2\} )^{-1}}{(\lambda_0 \odot \kappa)(ds,dx)} \\
&= \integral{B \setminus \{0\}}{}{(\min\{1,\norm{x}^2\} )^{-1}}{Q_0^{*}(A,dx)} \\
&= \phi_A(B).
\end{align*}
The uniqueness of $\rho_M$ follows by the Radon-Nikod\'{y}m theorem after countably many unions of null sets by considering the generator $\{M_1 \times \cdots \times M_m : M_j \in \mathcal{M} \}$ of $\B(\R^m)$ with
\begin{equation*}
\mathcal{M}:= \{ \{0\} \cup (-\infty,q_1] \cup [q_2,\infty) : q_1 \in \Q_{<0}, q_2 \in \Q_{>0} \}.
\end{equation*}
Conversely, the assumptions in (ii) ensure that $\phi_A(B):=\int_A \rho(s,B) \, \lambda(ds)$ with
\begin{equation*}
\integral{\R^m}{}{\min \{1, \norm{x}^2\} }{\phi_A(dx)}=\integral{A}{}{\integral{\R^m}{}{\min \{1,\norm{x}^2\} }{\rho(s,dx)}}{\lambda(ds)} \le \lambda(A)
\end{equation*}
is a L\'{evy} measure on $\R^m$ for every $A \in \mathcal{S}$, whereas the total variation of
\begin{equation*}
\mathcal{S} \ni A \mapsto \gamma_A:= \left (\lambda(A) - \integral{\R^m}{}{\min \{1,\norm{x}^2\}}{\phi_A(dx)} \right) e_1
\end{equation*}
is given by the non-negative expression in brackets for every $A \in \mathcal{S}$ (notice \eqref{eq:07031701} again). Here $e_j$ generally denotes the $j$-th unit vector. Now we can obviously use \cref{09121603} for the triplets $[\gamma_A,0,\phi_A]$ to obtain the assertion.
\end{proof}
\begin{prop} \label{29031701}
Let $M$ be an $\R^m$-valued ISRM on $\mathcal{S}$ with $M(A) \sim [\gamma_A,Q_A,\phi_A]$ for $A \in \mathcal{S}$.
\begin{itemize}
\item[(i)] There are $\sigma(\mathcal{S})$-measurable mappings $\alpha_M:S \rightarrow \R^m$ and $\beta_M: S \rightarrow \Li{m}$ such that the following integrals exist (component-wise) with
\begin{equation} \label{eq:08031701}
\integral{A}{}{\alpha_M(s)}{\lambda_M(ds)}=\gamma_A, \qquad \integral{A}{}{\beta_M(s)}{\lambda_M(ds)}=Q_A
\end{equation}
for every $A \in \mathcal{S}$. $\alpha_M$ and $\beta_M$ are uniquely determined $\lambda_M$-a.e. by \eqref{eq:08031701}.
\item[(ii)] $\beta_M(s)$ is symmetric and positive semi-definite $\lambda_M$-a.e.
\item[(iii)] The mapping
\begin{equation} \label{eq:08031702}
\R^m \ni t \mapsto \integral{A}{}{K_M(t,s)}{\lambda_M(ds)}
\end{equation}
is the log-characteristic function of $M(A)$ for every $A \in \mathcal{S}$, where $K_M:  \R^m \times S \rightarrow \C$ is defined by 
\begin{equation} \label{eq:18081701}
K_M(t,s)=\im \skp{\alpha_M(s)}{t} - \frac{1}{2} \skp{\beta_M(s) t}{t} + \integral{\R^m}{}{ \left(\text{e}^{\im \skp{t}{x}} -1- \frac{\im \skp{t}{x}}{1+\norm{x}^2}   \right)}{\rho_M(s,dx)}.
\end{equation}
\end{itemize}
\begin{proof}
\begin{itemize}
\item[(i)] We start with a general observation: Consider $T:\mathcal{S} \rightarrow \R$ $\sigma$-additive, then $|T|$ can be uniquely extended to a $\sigma$-finite measure $\widehat{|T|}$ where we assume that $\widehat{|T|} \ll \lambda_M$. Hence the same holds for the extensions $\widehat{T^{+}}$ of $T^{+}$ and $\widehat{T^{-}}$ of $T^{-}$ such that the Radon-Nikod\'{y}m theorem provides measurable, $[0,\infty]$-valued mappings $f^{\pm}$ with $\widehat{T^{\pm}}(A)=\int_A f^{\pm}(s) \, \lambda_M(ds)$ for $A \in \sigma(\mathcal{S})$. Choose $(S_n) \subset \mathcal{S}$ disjoint with $\cup_{n=1}^{\infty} S_n=S$. Then $f^{+} \indikatorzwei{S_n}$ and $f^{-} \indikatorzwei{S_n}$ are finite $\lambda_M$-a.e. Hence there are $\lambda_M$- null sets $N^{+}$ and $N^{-}$ such that $f^{+} \indikatorzwei{N^{+}}$ and $f^{+} \indikatorzwei{N^{+}}$ are finite, preserving the integral relation above instead of $f^{\pm}$, respectively. Then $f:=f^{+} \indikatorzwei{N^{+}}-f^{+} \indikatorzwei{N^{+}}$ is $\lambda_M$-integrable over every set $A \in \mathcal{S}$ with value $T(A)$. Thus the mappings $\alpha_M$ and $\beta_M$ can be obtained by using the previous method for each of its components, where $|Q| \le \lambda_M$ (on $\mathcal{S}$) and therefore $\widehat{|Q|} \ll \lambda_M$, which can be shown similarly as in the proof of \cref{20031705}.
\item[(ii)] In view of \cref{15031702} we observe that $A \mapsto \skp{Q_{A \cap S_n}x}{x}$ is a finite measure on $\sigma(\mathcal{S})$  while the Cauchy-Schwarz inequality yields that this measure is also absolutely continuous w.r.t $\lambda_M$. At the same time we know by (i) that $\skp{\beta_M(\cdot)x}{x}\indikator{S_n}{\cdot}$ is a corresponding $\lambda_M$-derivative which has to be non-negative $\lambda_M$-a.e. due to the Radon-Nikod\'{y}m theorem. Therefore we have $\skp{\beta_M(\cdot) x}{x} \ge 0$ except a $\lambda_M$-null set and for all $x \in \Q^m$, which finally means that $\beta_M(\cdot)$ is positive semi-definite $\lambda_M$-a.e. by continuity of the inner product. The symmetry follows if we consider the components $Q^{i,j}$ of $Q$. In particular we see that $A \mapsto (Q^{i,j}_{A \cap S_n}-Q^{j,i}_{A \cap S_n} )$ equals the zero measure on $\sigma(\mathcal{S})$ for every $n \in \N$ as $Q_{A \cap S_n}$ is symmetric.
\item[(iii)] The $\lambda_M$-integrability of $K_M(t,\cdot)$ and \eqref{eq:08031702} are almost obvious (see (i) and remember that $M(A) \sim [\gamma_A,Q_A,\phi_A]$). Using \cref{08031703} and \eqref{eq:06031702} it is easy to see that the following integral exists.
\begin{align*}
 \integral{A}{}{ \integral{\R^m}{}{ h(t,x) }{\rho_M(s,dx)}  }{\lambda_M(ds)} &= \integral{S \times \R^m}{}{h(t,x)  \indikator{A}{s} }{(\lambda_M \odot \rho_M)(ds,dx)} \\
&= \integral{\R^m}{}{h(t,x) }{\phi_A(dx)},
\end{align*}
where the last step is similar as before and $h(t,x)$ denotes the integrand used in the definition of $K_M$. 
\end{itemize}
\end{proof}
\end{prop}
\begin{remark} \label{20081701}
In view of \eqref{eq:08031702} and the uniqueness of the L\'{e}vy-Khintchine-Formula we write $M \sim ( \lambda_M, K_M)$. And in the case of $\alpha_M=\beta_M=0$ we may even write $M \sim ( \lambda_M, \rho_M)$, respectively. Observe that the latter case applies to \cref{08031703} (ii) as long as \eqref{eq:07031701} holds with equality.
\end{remark}
\begin{example} \label{21031706}
\begin{itemize}
\item[(a)] Consider a $\sigma$-finite measure space $(S,\Sigma,\nu)$ and let $\mu \sim [\gamma',Q',\phi']$ be an i.d. distribution on $\R^m$ with log-characteristic function $\psi$ and not being the point measure at zero. Then $\mathcal{S}_{\nu}:=\{A \in \Sigma: \nu(A)< \infty \}$ is a $\delta$-ring with $\sigma(\mathcal{S}_{\nu})=\Sigma$ which can be verified easily with the aid of $(S_n)$. Hence, according to \cref{09121603}, there exists an i.d. ISRM $M$ with $M(A) \sim [\nu(A) \cdot \gamma',\nu(A) \cdot Q',\nu(A) \cdot \phi']$ for every $A \in \mathcal{S}_{\nu}$ and we say that $M$ is \textit{generated by $\nu$ and $\mu$}. Moreover, with
\begin{equation*}
C_{\mu}:= \norm{\gamma'}+\text{tr}(Q') +\integral{\R^m}{}{\min \{1,\norm{x}^2\} }{\phi'(dx)} \quad \in (0,\infty)
\end{equation*}
we get that $\lambda_M(\cdot)= C_{\mu} \cdot \nu(\cdot)$, whereas $\rho_M(\cdot)=C_{\mu}^{-1} \cdot \phi'(\cdot)$ and $K_M(\cdot)=C_{\mu}^{-1} \cdot \psi(\cdot)$ are both constant in $s \in S$. Therefore it is convenient to write $M \sim (\nu,\mu)$ and one can check by the construction in \cref{09121603}: $M(A_1)$ and $M(A_2)$ are independent if and only if $\nu(A_1 \cap A_2)=0$. Furthermore, independence of $M(A_1),...,M(A_n)$ is equivalent to pairwise independence.
\item[(b)] In \cite{falconer3} an $\R$-valued ISRM $M_{\alpha}$ is constructed such that the log-characteristic function of $M_{\alpha}(A)$ is given by 
\begin{equation} \label{eq:13031701}
\R \ni t \mapsto - \integral{A}{}{|t|^{\alpha(s)}}{ds}
\end{equation}
for every Borel set $A \subset \R$ with finite Lebesgue measure. Here $\alpha:\R \rightarrow [a,b]$ is a measurable function with $0<a \le b< 2$ and $M$ is called an \textit{$\alpha(s)$-multistable random measure}. On the one hand \cref{09121603} says that $M_{\alpha}$ is uniquely determined by \eqref{eq:13031701}, on the other hand $M$ can be recovered by our approach and \eqref{eq:08031702}: Denote by $\rho_{\alpha}(s, \cdot)$ for every $s \in \R$ the Borel measure with Lebesgue density $x \mapsto \theta(s) \, |x|^{-\alpha(s)-1}$, where $\theta(s):=\frac{\alpha(s)}{4}  (2-\alpha(s)) \in [c_1,c_2] $ for all $s \in \R$ and suitable $0<c_1 \le c_2 < \infty$  by the assumption on $\alpha(s)$, i.e. \eqref{eq:07031701} is fulfilled with equality. Similarly and as in \cite{SaTaq94} there exists a measurable function $\eta:\R \rightarrow [c_3,c_4] \subset (0,\infty)$ such that
\begin{equation*} \label{eq:09051702}
\eta(s) \integral{\R}{}{\left( \text{e}^{\im tx}-1-\frac{\im t x}{1+x^2} \right)|x|^{-\alpha(s)-1}}{dx}=- |t|^{\alpha(s)}
\end{equation*}
for every $s,t \in \R$. Finally let $\lambda_{\alpha}(\cdot)$ be the Borel measure with Lebesgue density $s \mapsto (\theta(s) \eta(s))^{-1}$ and apply \cref{08031703}, that means $M_{\alpha} \sim (\lambda_{\alpha},\rho_{\alpha})$ by \cref{20081701}.
\end{itemize}
\end{example}
\begin{remark} \label{19041701}
If we identify $\B(\C^m)$ and $\B(\R^{2m})$ by means of $\Xi$, we can observe that the relation between i.d. random measures with values in $\C^m$ and $\R^{2m}$, respectively, is one-to-one. Generally, for any $\C^m$-valued ISRM $M$, we say that $\Xi(M)$ is its \textit{real associated} ISRM. \\
Of course, we can (and will do) interpret every $\R^m$-valued i.d. ISRM $M$ as such a one with values in $\C^m$, having no imaginary parts which leads to $\Xi(t):=(t,0)$ for every $t \in \R^m$. Hence, in this case we understand $\Xi$ as a mapping with domain $\R^m$. Furthermore, we then see that $\Xi(M)(A)\sim [\tilde{\gamma}_A, \tilde{Q}_A, \tilde{\phi}_A]$ with 
\begin{center}
$\tilde{\gamma}_A=(\gamma_A,0)$, $\tilde{Q}_A=\begin{pmatrix} Q_A & 0 \\ 0 & 0 \end{pmatrix}$ and $ \tilde{\phi}_A=\Xi (\phi_A), \quad A \in \mathcal{S}.$
\end{center}
Similarly, this works for the objects in \cref{29031701} and one immediately checks that $\lambda_{\Xi(M)}=\lambda_M$, whereas the transition function becomes $\rho_{\Xi(M)}(s,A)=\rho_M(s,\Xi^{-1}(A))$ for any $A \in \B(\R^{2m})$ together with $K_M(s,t_1)=K_{\Xi(M)}(s,t)$ for all $s \in S$ and $t=(t_1,t_2) \in \R^{2m}$. 
\end{remark}
%% Section 4: Atomless RM
\section{Atomless random measures}
Throughout this chapter we denote by $M$ some fixed ISRM on a $\delta$-ring $\mathcal{S}$ with values in $\R^m$. Following \cite{Pre57} we call a set $A \in \mathcal{S}$ \textit{crucial} if
\begin{center}
$M(A \cap B)=0$ a.s. \quad or \quad  $M(A \cap B)=A$ a.s.
\end{center}
is true for every $B \in \mathcal{S}$. Then $M$ itself is called \textit{atomless}, if we have $M(A)=0$ a.s. or equivalently $\lambda_M(A)=0$ for every crucial set $A \in \mathcal{S}$. Conversely, any crucial set $A$ with $\lambda_M(A)>0$ is called an \textit{atom} of $M$. This definition appears even more natural in the light of the following statement, which is also similar to \cite{Pre57} and where the underlying probability space is still $(\Omega, \mathcal{A},\Pro)$.
\begin{prop} \label{25041703}
$M$ is atomless if and only if the following implication holds for every $A \in \mathcal{S}$ with $\Pro(M(A) \ne 0)>0$:
\begin{equation} \label{eq:25041701}
\exists A_1,A_2 \in \mathcal{S}, A_1 \cap A_2 = \emptyset:\quad  \Pro(M(A \cap A_i) \ne 0)>0 \quad (i=1,2).
\end{equation}
\end{prop}
\begin{proof}
Assume that there is an $A \in \mathcal{S}$ with $\Pro(M(A) \ne 0)>0$ such that \eqref{eq:25041701} is false. Then, because of $M(A)=M(A\cap B)+ M(A \cap (A \setminus B))$ a.s. for every $B \in \mathcal{S}$, this implies that $A$ is crucial, which contradicts the assumption as long as $M$ is atomless. \\
Conversely, if we assume that $M$ has an atom $A$, \eqref{eq:25041701} provides sets $A_1,A_2$ as mentioned above which necessarily leads to $M(A \cap A_1)=M(A)=M(A \cap A_2)$ a.s. Use again that $A$ is an atom together with $A \setminus C=A \cap (A \setminus C)$ for every $C \in \mathcal{S}$ to check that we have $M(A \setminus (A_1 \cup A_2))=0$ a.s. or $M(A \setminus (A_1 \cup A_2))=M(A)$ a.s., respectively. Finally, we can combine both findings and obtain with $(RM_2)$ that there is a $k \in \{2,3\}$ such that $M(A)=k \cdot M(A)$ a.s. which easily provides the contradiction.
\end{proof}
\begin{remark} \label{14031702}
Obviously we can also use the previous definition and proposition for deterministic measures which means that the corresponding \textit{probabilities} are $\{0,1\}$-valued.
\end{remark}
Now we want to formulate the central result of this section which will be false if we relax the definition of atomless random measures (as in \cite{Woy70} for example), since we are considering general $\delta$-rings. Observe likewise that the converse of the following theorem cannot hold either.
\begin{theorem} \label{14031704}
If $M$ is atomless, then it is i.d.
\end{theorem}
The proof requires some preparation. Therefore let $X$ be an arbitrary $\R^m$-valued random vector and denote its characteristic function by $\varphi$. 
\begin{lemma} \label{25041705}
For every $\delta>0$ there exists a $C'(\delta)>0$ such that
\begin{equation*}
\Pro(\norm{X} \ge \delta) \le C'(\delta) \integral{[-\delta,\delta]^m}{}{(1-\varphi(t))}{dt}.
\end{equation*}
\end{lemma}
\begin{proof}
Define $g(y):=\sin(y)/y$ for $y \ne 0$ and $g(0):=1$. Then simple calculations show that for given $\delta,\gamma>0$ there exists a $C(\delta, \gamma) \in (0,1)$ such that $1- \prod_{j=1}^m g(\delta x_j) \ge C(\delta,\gamma)$ for every $x=(x_1,...,x_m)$ with $\norm{x} \ge \gamma$. After this a multivariate extension of (1.2) in \cite{Pre56} yields the assertion, where we can choose $C'(\delta)=((2 \delta)^m C(\delta,\delta) )^{-1}$.
\end{proof}
\begin{lemma} \label{25041702}
$\Pro(X \ne 0)>0$ implies
\begin{equation*} 
h(X,T):= \sup \{|1-\varphi(t)|: \norm{t}_{\infty} \le T\}>0
\end{equation*}
for all $T>0$. Conversely, if $h(X,T)>0$ for some $T>0$, then we have $\Pro(X \ne 0)>0$.
\end{lemma}
\begin{proof}
Assume that there exists a $T>0$ with $h(X,T)=0$. Then, with the use of
\begin{equation*}
0 \le 1- |\varphi(2t)|^2 \le 4 (1-|\varphi(t)|^2), \quad t \in \R^m
\end{equation*}
(see Proposition 1.3.4 in \cite{thebook}), we obtain $h(X,2T)=0$ which contradicts $\Pro(X \ne 0)>0$ by induction. The converse is obvious.
\end{proof}
We return to $M$ and denote the characteristic function of $M(A)$ by $f(\cdot,A)$. Then we define $g_T:\mathcal{S} \rightarrow [0,2]$ via
\begin{equation*}
g_T(A):= \sup \{|1-f(t,A)| : \norm{t}_{\infty} \le T\}=h(M(A),T), \quad A \in \mathcal{S}
\end{equation*}
for every $T>0$. Unfortunately, $g_T$ will not be $\sigma$-additive in general, therefore we have to consider its total variation $|g_T|$. However, the fact that $|g_T|$ can be infinite prohibits a direct application of Lemma 1 in \cite{Pre56-2}. Also note that the following statement is in part similar to Theorem 2.1 in \cite{Pre57}.
\begin{prop} \label{26041701}
$|g_T|$ is $\sigma$-additive for every $T>0$ (with values in $[0,\infty]$). Furthermore:
\begin{itemize}
\item[(a)] If $M$ is atomless, then $|g_T|$ is atomless for every $T>0$ (in the sense of \cref{14031702}).
\item[(b)] If $|g_T|$ is atomless for some $T>0$, then the same holds for $M$ itself.
\end{itemize}
\begin{proof}
Inductively we see that $|1-\prod_{j=1}^n z_j| \le \sum_{j=1}^{n} |1-z_j|$ for any $n \ge 1$ and complex numbers with $|z_j| \le 1$. Hence  $|1-f(t,A)| \le \sum_{j=1}^{\infty} |1-f(t,A_j)|$ by L\'{e}vy's continuity theorem for every $t \in \R^m$ and any disjoint sequence $(A_j) \subset \mathcal{S}$ with $A:=\cup_{j=1}^{\infty} A_j \in \mathcal{S}$. Therefore $g_T$ is $\sigma$-subadditive and \cref{24041701} gives the first part of the assertion. For (a) observe that $|g_T|(A)>0$ always implies the existence of a set $ A' \subset A$ in $\mathcal{S}$ with $g_T(A')>0$ such that $\Pro(M(A')\ne 0)>0$ due to the previous Lemma. Now we can first use \eqref{eq:25041701} (which yields appropriate sets $A_1,A_2$) and then \cref{25041702} again to see that $g_T(A' \cap A_i)>0$. Especially $|g_T|(A \cap (A' \cap A_i))>0$ for $i=1,2$ which gives the statement by \cref{25041703} and \cref{14031702}. The proof of (b) is similar and therefore left to the reader.
\end{proof}
\end{prop}
\begin{proof}[Proof of \cref{14031704}] In the following assume that $S_n \in \mathcal{S}$ are disjoint with $\cup_{n=1}^{\infty} S_n=S$. \\
\underline{First step:} For $X$ as above a straight-forward extension of Theorem 3.1 in \cite{Pre56}, using the Cauchy-Schwarz inequality and Lemma 8.6 in \cite{Sato} for example, yields to
\begin{align*}
|1-\omega(t)| & \le \frac{\norm{t}^2}{2} Var (h(X)) + \frac{\norm{t}^2}{2} \erwartung{h(X)}^2 + \norm{t} \erwartung{h(X)}+ 2 \Pro(\norm{X}>1) \\
& \le \frac{\norm{t}^2}{2} Var (h(X)) + \left (\frac{\norm{t}^2}{2} + \norm{t} \right) \erwartung{h(X)} + 2 \Pro(\norm{X}>1),
\end{align*}
for every $t \in \R^m$, where $h(y):=\norm{y}Ê\cdot \indikatorzwei{\norm{y} \le 1}$. Due to Theorem 15.50 in \cite{Kle08} this shows that if $(X_n)$ is a sequence of independent $\R^m$-valued random vectors with $\sum_{n=1}^{\infty} \norm{X_n}< \infty$ a.s., then we have convergence of each of the following series:
\begin{equation*}
\summe{n=1}{\infty} \Pro(\norm{X_n}>1) , \quad \summe{n=1}{\infty} \erwartung{h(X_n)} , \quad  \summe{n=1}{\infty} Var(h(X_n)) .
\end{equation*}
Denoting the characteristic function of $X_n$ by $\varphi_n(\cdot)$, we can combine both findings to see that the series $\sum_{n=1}^{\infty} \sup \{|1-\varphi_n(t)| : \norm{t}_{\infty} \le T\}$ converges for every $T>0$ in this case. \\
\underline{Second step:} \cref{15031702} allows to define the $[0,2]$-valued mapping $g_T^{(n)}(A):=g_T(A \cap S_n)$ on $\sigma(\mathcal{S})$ which inherits the $\sigma$-subadditivity. Fix $n \in \N$ and $T>0$ as well as some disjoint sequence $(A_k) \subset \sigma(\mathcal{S})$. Then the union over $(A_k \cap S_n)_k$ belongs to $\mathcal{S}$ (see \eqref{eq:21161101}) such that the series $\sum_{k=1}^{\infty} M(A_k \cap S_n)$ converges a.s., namely absolutely due to $(RM_2)$. Hence, together with $(RM_1)$ the first step can be applied to obtain
\begin{equation*}
\summe{k=1}{\infty} \sup \{|1-f(t,A_k \cap S_n|) : \norm{t}_{\infty} \le T \} =\summe{k=1}{\infty}g_T^{(n)} (A_k) < \infty.
\end{equation*}
Then Theorems 1.1 and 1.2 in \cite{Pre56} imply that $|g_T^{(n)}|$ is a finite measure on $\sigma(\mathcal{S})$. Indeed, this leads besides \cref{26041701} to the fact that $|g_T|^{(n)}(A):=|g_T|(A \cap S_n)$ also defines a finite measure on $\sigma(\mathcal{S})$ as we have $|g_T|^{(n)} \le |g_T^{(n)}|$. Consider $A \in \sigma(\mathcal{S})$ arbitrary, then the latter claim is clear by definition of the total variation since every $B \subset (A \cap S_n)$ fulfills $B=B \cap S_n$  as well as $B \subset A$. \\
\underline{Third step:} Fix some $A \in \mathcal{S}$ arbitrary. Using the idea of Theorem 2.2 in \cite{Pre57} we can construct a sequence of families $\{C_0^{(l)},C_1^{(l)},...,C_{k(l)}^{(l)}\}$ such that the following holds for every $l \in \N$:
\begin{itemize}
\item[(i)] $C_0^{(l)},...,C_{k(l)}^{(l)}$ belong to $\mathcal{S}$ and are disjoint.
\item[(ii)] $C_0^{(l)} \cup  \cdots \cup C_{k(l)}^{(l)}=A$.
\item[(iii)] $\Pro(\norm{M(C_0^{(l)})} \ge 1/l) \le 1/l.$
\item[(iv)] $|g_{1/l}|(C_j^{(l)}) \le \eps_l$ for $j=1,...,k(l)$ with $\eps_l:= l^{m-1} 2^{-m} C'(1/l)^{-1}>0$, where $C'(1/l)$ is as in \cref{25041705}.
\end{itemize}
Fix $l \in \N$ arbitrary and check with $(RM_2)$ that $M(\cup_{n=\nu+1}^{\infty}(A \cap S_n)) \rightarrow 0$ a.s. In particular we can find $\nu(l)$ sufficiently large such that (iii) is fulfilled for $C_0^{(l)}:= \cup_{n=\nu(l)+1}^{\infty} (A \cap S_n)$. Next we consider $A \cap S_1$ and the measure $|g_{1/l}|^{(1)}$ which is finite according to the previous step. Hence IV 9, Lemma 7 in \cite{DunSch57} provides finitely many disjoint sets $D_1^{(l)},...,D_{(k,1)}^{(l)} \in \sigma(\mathcal{S})$ whose union equals $S$ and where $D_j^{(l)}$ is either an \textit{atom} or fullfils $|g_{1/l}|^{(1)} (D_j^{(l)}) \le \eps_l$ for $j=1,...,k(1,l)$. One can check easily that the definition for an atom in \cite{DunSch57} leads to the latter conclusion as we assume $M$ to be atomless. Similarly we obtain disjoint sets $D_{k(1,l)+1}^{(l)},....,D_{k(2,l)}^{(l)} \in \sigma(\mathcal{S})$ that exhaust $S$ with $|g_{1/l}|^{(2)}(D_j^{(l)}) \le \eps_l$ for $j=k(1,l)+1,...,k(2,l)$. Continue this procedure until the consideration of $|g_{1/l}|^{(\nu(l))}$, leading to $D_1^{(l)},...,D_{k(l)}^{(l)}$ with $k(l)=k(\nu(l),l)$. This obviously completes the construction via 
\begin{equation*}
C_j^{(l)}:=D_j^{(l)} \cap A \cap S_n, \quad \text{if $k(n-1,l)<j \le k(n,l)$}
\end{equation*}
for $j=1,...,k(l)$ with $k(0,l):=0$. \\
\underline{Fourth step:} We have seen that $M(A)= \sum_{j=0}^{k(l)} M(C_j^{(l)})$ holds a.s. for any $l \in \N$. Thus due to (i) this defines a triangular array $\Gamma:=\{M(C_j^{(l)}): 0 \le j \le k(l), l \in \N\}$ in the sense of Definition 3.2.1 in \cite{thebook} and we can assume that $k(1) \ge 1$ as well as that $k(l+1)>k(l)$. Furthermore, a simple calculation and the definition of $g_{T} / |g_T|$ show that the the statement of (iii) can be extended for every $C_j^{(l)}$ ($j=0,...,k(l)$) thanks to \cref{25041705} and the choice of $\eps_l$. Therefore $\Gamma$ is \textit{infinitesimal} and Theorem 3.2.14 in \cite{thebook} completes the proof, i.e. $M(A)$ is i.d.
\end{proof}
%% Section 5: Stochastic Integral
\section{Integrals with respect to ISRMs}
Let $M$ be a $\mathbb{K}^m$-valued ISRM on a $\delta$-ring $\mathcal{S}$, where we assume that $M$ is i.d. Then a matrix-valued mapping $f: S \rightarrow$ L$(\mathbb{K}^m)$ is called \textit{$\mathcal{S}$-simple}, if $f$ can be represented by $f=\sum_{j=1}^{n} R_j \indikatorzwei{A_j}$ with $R_1,...,R_n \in$ L$(\mathbb{K}^m)$ and $A_1,...,A_n \in \mathcal{S}$ disjoint. In this case we define the stochastic integral of $f \indikatorzwei{A}$ w.r.t $M$ by
\begin{equation} \label{eq:15031701}
I_M(f  \, \indikatorzwei{A}):= I(f  \, \indikatorzwei{A}) := \integral{A}{}{f}{dM}:=\integral{A}{}{f(s)}{M(ds)}:= \summe{j=1}{n} R_j \, M(A \cap A_j).
\end{equation}
Note that, in view of \cref{15031702}, the mentioned truncation is valid for every $A \in \sigma(\mathcal{S})$ and that the stochastic integral is well-defined a.s. by $(RM_2)$. Write $I_M(f)$ and so on for $A=S$.
\begin{defi} \label{20031706}
Let $f:S \rightarrow$ L$(\mathbb{K}^m)$ be $\sigma(\mathcal{S})$-$\B($L$(\mathbb{K}^m))$-measurable. 
\begin{itemize}
\item[(a)] $f$ is called \textit{$M$-integrable}, if there exists a sequence $(f_n)$ of $\mathcal{S}$-simple functions such that the following conditions hold:
\begin{itemize}
\item[($I_1$)] $f_n \rightarrow f$ pointwise $\lambda_M/ \lambda_{\Xi(M)}$-a.e. for $\mathbb{K}=\R / \C$.
\item[($I_2$)] The sequence $I(f_n \indikatorzwei{A})$ converges in probability for every $A \in \sigma(\mathcal{S})$ and we refer to this limit as $I_M(f \indikatorzwei{A})$ or any synonymous notation from \eqref{eq:15031701}, respectively.
\end{itemize}
\item[(b)] Consider $\mathbb{K}=\C$. If we relax $(I_2)$ in such a way that we merely want either the sequences $\text{Re } I(f_n \indikatorzwei{A})$ or the sequences $\text{Im } I(f_n \indikatorzwei{A})$ to converge for every $A \in \sigma(\mathcal{S})$, then $f$ is called \textit{partially $M$-integrable (in the real/imaginary sense)}. 
\end{itemize}
Finally we define
\begin{equation*}
\mathcal{I}_{(p)}(M):=\{f: (S,\sigma(\mathcal{S})) \rightarrow (\text{L}(\mathbb{K}^m), \B(\text{L}(\mathbb{K}^m))) \, | \, \text{$f$ is (partially) $M$-integrable} \}.
\end{equation*}
\end{defi}
\begin{remark} \label{20031703}
\begin{itemize}
\item[(i)] The previous definition coincides with \eqref{eq:15031701} for simple $f$, whereas the notation in $(I_2)$ will be justified by \cref{21081701} (a).
\item[(ii)] If the imaginary parts of $f$ and $M$ vanish, we get back the case $\mathbb{K}=\R$. 
\item[(ii)] The two types of partial integrability only differ in the consideration of $f$ and $- \im f$. Hence we restrict to partial integrability in the real sense and write $\text{Re } I_M(f \indikatorzwei{A})$ for the corresponding limit in $(b)$, even if $I_M (f \indikatorzwei{A})$ may not exist in the sense of $(a)$. However we have $\mathcal{I}(M) \subset \mathcal{I}_{(p)}(M)$, generally with non-equality.
\end{itemize}
\end{remark}
Now we state some useful properties, starting with the linearity which essentially illuminates the notation \textit{(stochastic) integral}. Throughout and for accuracy we should identify random vectors that are identical a.s. Also notice that ${}^{*}$ denotes the adjoint operator in the Hermitian sense.
\begin{prop} \label{04051701}
Let $M$ be as before. Then we have:
\begin{itemize}
\item[(a)] $\mathcal{I}(M)$ is a $\mathbb{K}$-vector space and the mapping $\mathcal{I}(M) \ni f \mapsto I_M(f)$ is linear a.s.
\item[(b)] $f \in \mathcal{I}(M)$ implies that for every $Q \in$ L$(\mathbb{K}^m)$ the function $Q \cdot f$, defined by $(Q \cdot f) (s)=Q f(s)$, also belongs to $\mathcal{I}(M)$ with $I_M(Q\cdot f) = Q I_M(f)$ a.s. 
\end{itemize}
Both statements hold accordingly for $\mathcal{I}_p(M)$ with $\mathbb{K}=\R$.
\end{prop}
\begin{proof}
The linearity in (a) is obvious for simple functions when considering a common partition $A_1,...,A_n \in \mathcal{S}$ and extends for general $f,g \in \mathcal{I}(M)$ (with $\mathcal{S}$-simple approximating sequences $(f_n)$ an $(g_n)$) since $h_n:= \alpha f_n+ \beta g_n$ approximates $h:=\alpha f + \beta g$ properly for $\alpha_1, \alpha_2 \in \mathbb{K}$. Merely note in the case of $\mathbb{K}=\C$ that 
\begin{equation*}
\text{Re } I_M(h_n \indikatorzwei{A}) =x_1 \text{Re } (f_n \indikatorzwei{A})- y_1 \text{Im } (f_n \indikatorzwei{A}) + x_2 \text{Re } (g_n \indikatorzwei{A})- y_2 \text{Im } (g_n \indikatorzwei{A}), \quad A \in \sigma(\mathcal{S}),
\end{equation*}
if $\alpha_i=x_i + \im y_i$; similarly for the imaginary parts. In particular we get $h \in \mathcal{I}(M)$ by additivity of the stochastic limit which implies that $\mathcal{I}(M)$ is a vector space. Part (b) and the additional statement for $\mathcal{I}_p(M)$ can be proven quite similarly. 
\end{proof}
For the time being we consider the case $\mathbb{K}=\R$. Recall from \eqref{eq:09121601} and \eqref{eq:18081701} the definition of $\lambda_M$ and $K_M$, respectively.
\begin{theorem} \label{21081701}
Let $M$ be as before. 
\begin{itemize}
\item[(a)] If $f \in \mathcal{I}(M)$, then $I_M(f \indikatorzwei{A})$ is i.d. for every $A \in \sigma(\mathcal{S})$ and its log-characteristic function is given by
\begin{equation} \label{eq:21081701}
\R^{m} \ni t \mapsto \integral{A}{}{ K_{M} ( f(s)^{*}t ,s)}{\lambda_{M}(ds)}. 
\end{equation}
Particularly the integral in \eqref{eq:21081701} exists and $I_M(f \indikatorzwei{A})$ is well-defined a.s.
\item[(b)] If $f_1,...,f_n \in \mathcal{I}(M)$, then we have for any $t_1,...,t_n \in \R^{m}$:
\begin{equation*}
\bigerwartung{  \text{e}^{\im \summezwei{j=1}{n} \skp{I(f_j)}{t_j}  }} = \exp \left( \, \integral{S}{}{K_{M} \left( \summezwei{j=1}{n} f_j(s)^{*} t_j ,s \right)}{\lambda_{M}(ds)}   \right).
\end{equation*}
\item[(c)] For $f_,f_1,f_2,... \in \mathcal{I}(M)$ we have that $I_M(f_n) \rightarrow I_M(f)$ in probability is equivalent to 
\begin{equation} \label{eq:21081703}
\integral{\R^m}{}{K_{M}(  (f_n(s)-f(s))^{*} t ,s)}{\lambda_{M}(ds)} \rightarrow 0, \quad t \in \R^{m}.
\end{equation}
\item[(d)] Let $f_1,f_2 \in \mathcal{I}(M)$ such that $\norm{f_1(s)} \cdot \norm{f_2(s)}=0$ holds $\lambda_{M}$-a.e. Then $I_M(f_1)$ and $I_M(f_2)$ are independent.
\end{itemize}
\end{theorem}
\begin{proof}
For simple $f$, one checks that $I_M(f \indikatorzwei{A})$ is i.d. (see Proposition 3.1.21 in \cite{thebook}) for every $A \in \sigma(\mathcal{S})$ while $K_M(0,\cdot)=0$ and \eqref{eq:08031702} yield that its characteristic function is given by \eqref{eq:21081701}. Note that $t \mapsto K_M(t,s)$ is the log-characteristic function of the distribution with triplet $[\alpha_M(s),\beta_M(s), \rho_M(s)]$, i.e. is continuous for every $s \in S$. On one hand this merely shows that the integral function in \eqref{eq:21081701} is really \textit{the} log-characteristic function of $I_M(f)$. On the other hand it allows us to perform a simple multivariate extension of Proposition 2.6 in \cite{RajRo89} which states that \eqref{eq:21081701} and the previous implication concerning the log-characteristic function also hold for general $f \in \mathcal{I}(M)$, namely the limit in $(I_2)$. This limit preserves the infinite divisibility and since the right-hand side in \eqref{eq:21081701} does not depend on the choice of approximating functions $(f_n)$, we see that $I_M(f \indikatorzwei{A})$ is uniquely determined a.s. after consideration of $(f_n-f_n')$, provided that $(f_n')$ also approximates $f$ properly. This immediately yields (a). The proof of (b) will be covered by the one in \cref{20031701} (b), while part (c) is a direct conclusion of (a), the linearity and Lemma 3.1.10 in \cite{thebook}. Finally for (d) we show that $\norm{f_1(s)} \cdot \norm{f_2(s)}=0$ expect a potential $\lambda_M$-null set implies the independence of $I_M(f_1)$ and $I_M(f_2)$. Define $A_i:=\{s:f_i(s) \ne 0\}$ $(i=1,2)$ and observe that $M(A)=0$ a.s. for every $A \subset(A_1 \cap A_2)$ by assumption and the use of \cref{20031705} (ii). Now if $(f_{n,i})$ is an approximating sequence of simple functions for $f_i$, we see that this also applies to $f_{n,i} \indikatorzwei{A_i}$ and that $I_M(f_{n,i} \indikatorzwei{A_i})=I_M(f_{n,i} \indikatorzwei{A_i \setminus (A_1 \cap A_2)})$ a.s. In view of $(RM_1)$ this gives the assertion. 
\end{proof}
In the following we are going to characterize the class $\mathcal{I}(M)$ for a given ISRM $M$ in terms of its control measure $\lambda_M$ and the related function $K_M$. Also recall the definition of $\alpha_M,\beta_M$ and $\rho_M$ in \cref{08031703} as well as in \cref{29031701} and define
\begin{align*}
&U_M:\text{L}(\R^m) \times S \rightarrow \R^m, \quad (R,s) \mapsto R\, \alpha_M(s)+ \integral{\R^m}{}{ \left( \frac{R \, x}{1+\norm{R\,x}^2} - \frac{R\,x}{1+\norm{x}^2} \right)}
{\rho_M(s,dx)} , \\
&V_M:\text{L}(\R^m) \times S \rightarrow \R_{+}, \quad (R,s) \mapsto  \integral{\R^m}{}{\min \{1, \norm{R\,x}^2\} }{\rho_M(s,dx)}.
\end{align*}
Recall that these functions are multivariate extensions of those in \cite{RajRo89} and a simple calculation shows that
\begin{equation} \label{eq:21031701}
\bignorm{ \frac{R \, x}{1+\norm{R\,x}^2} - \frac{R\,x}{1+\norm{x}^2}} \le \max \{2, \norm{R}+\norm{R}^3\} \, \min \{1,\norm{x}^2\}
\end{equation}
holds for all $R \in$ L$(\R^m)$ and $x \in \R^m$. Similarly and with the help of the Cauchy-Schwarz inequality we see that
\begin{equation} \label{eq:21031702}
\left| \frac{\skp{t}{y} }{1+\norm{y}^2}  - \sin \skp{t}{y} \right| \le (1+\norm{t}+\norm{t}^2) \min \{1,\norm{y}^2\}, \quad t,y \in \R^m.
\end{equation}
Observe that, in view of \eqref{eq:21031701}, $U_M$ exists. The following proposition is the first step in the promised characterization of $\mathcal{I}(M)$ and also provides the L\'{e}vy-Khintchine-Triplet of the i.d. random vector $I_M(f)$. But in contrast of the univariate case considered in \cite{RajRo89} in our situation the arguments are more involved.
\begin{prop} \label{27031702}
Assume that $f \in \mathcal{I}(M)$. Then the following integrals exist
\begin{equation*}
\gamma_{f}:= \integral{S}{}{U_M(f(s),s)}{\lambda_M(ds)}, \quad Q_{f}:=\integral{S}{}{f(s) \beta_m(s) f(s)^{*}}{\lambda_M(ds)}
\end{equation*}
and 
\begin{equation*}
\phi_{f}(A):= (\lambda_M \odot \rho_M) (\{(s,x) \in S \times \R^m: f(s)x \in A \setminus \{0\} \} ) , \quad A \in \B(\R^m)
\end{equation*}
defines a L\'{evy} measure. Moreover we have $I_M(f) \sim [\gamma_{f}, Q_{f}, \phi_{f}]$.
\end{prop}
\begin{proof}
The given assumption and \cref{21081701} (a) ensure the existence of
\begin{equation} \label{eq:21031707}
 \integral{S}{}{  K_M(f(s)^{*}t,s) }{ \lambda_M(ds)}
\end{equation}
for every $t \in \R^m$ as well as the continuity of 
\begin{equation} \label{eq:21031704}
\R^m \ni t \mapsto \integral{S}{}{ \text{Re } K_M(f(s)^{*}t,s) }{ \lambda_M(ds)}.
\end{equation}
Indeed, both statements will suffice to perform the present proof. \cref{29031701} (b) permits the following decomposition for every $t \in \R^m$ and the use of \eqref{eq:06031702} combined with the definition of $\phi_f$ yields
\allowdisplaybreaks
\begin{align*}
& \integral{S}{}{ \text{Re } K_M(f(s)^{*}t,s) }{ \lambda_M(ds)} \\
&= -  \integral{S}{}{\frac{1}{2}\skp{\beta_M(s)f(s)^{*}t}{f(s)^{*}t}  }{\lambda_M(ds)} - \integral{S}{}{ \integral{\R^m}{}{ (1 - \cos \skp{f(s)^{*}t}{x})}{\rho_M(s,dx)}}{ \lambda_M(ds)} \\
&= -  \integral{S}{}{\frac{1}{2} \skp{f(s) \beta_M(s)f(s)^{*}t}{t}  }{\lambda_M(ds)} - \integral{ \R^m}{}{ (1 - \cos \skp{t}{x})}{\phi_f(dx)}.
 \end{align*}
Now let $C(s):=f(s) \beta_M(s)f(s)^{*}$ with $C(s)=(C^{i,j}(s))_{i,j=1,...,m}$ and first consider $t=e_i$ to check the $\lambda_M$-integrability of the diagonal components $C^{i,i}$. Repeat this argument for $t=e_i+e_j$ for the $\lambda_M$-integrability of $C^{i,j}+C^{j,i}$ which finally gives the existence of $Q_f$ due to the symmetry in \cref{29031701} (b). Here we should also note that $Q_f$ is symmetric and positive semi-definite since $\beta_M$ is (at least $\lambda_M$-a.e.). In particular we know that 
\begin{align}
  \integral{\R^m}{}{ (1 - \cos \skp{t}{x})}{\phi_f(dx)}=  -\frac{1}{2} \skp{Q_f t}{t} -\integral{S}{}{\text{Re } K(f(s)^{*}t,s)}{\lambda_M(ds)}, \quad t \in \R^m. \label{eq:13071701}
\end{align}
Hence the left-hand side is continuous in $t$ according to \eqref{eq:21031704}, i.e. $\phi_f$ is a L\'{e}vy measure, if we include $\phi_f(\{0\})=0$ and  perform similar steps as done in the proof of Theorem 3.3.10 in \cite{Ro87}. Then we can argue as above that this implies the $\lambda_M$-integrability of $V_M(f(\cdot),\cdot)$. For the existence of $\gamma_f$ it finally suffices to show that $\skp{t}{U_M(f(\cdot),\cdot)}$ is $\lambda_M$-integrable for every $t \in \R^m$. Observe that we have the decomposition 
\begin{equation*}
\skp{t}{U(f(s),s)} = \text{Im } K_M(f(s)^{*}t,s)+  \integral{\R^m}{}{\left(\frac{\skp{t}{f(s)x}}{1+\norm{f(s)x}^2}   - \sin \skp{t}{f(s) x}  \right)}{\rho_M(s,dx)}  
\end{equation*} 
for every $s \in S,t \in \R^m$ in view of \eqref{eq:21031701} and \eqref{eq:21031702}. Furthermore, \eqref{eq:21031702} implies that
\begin{align*}
\integral{S}{}{|\skp{t}{U(f(s),s)}| }{\lambda_M(ds)} & \le \integral{S}{}{|K(f(s)^{*}t,s)|}{\lambda_M(ds)} + C(t) \integral{S}{}{V(f(s),s)}{\lambda_M(ds)}< \infty
\end{align*}
with $C(t):=1+ \norm{t}+\norm{t}^2$ and because of what we have shown before. Now it is easy to see that $I_M(f) \sim [\gamma_{f}, Q_{f}, \phi_{f}]$.
\end{proof}
\begin{lemma}  \label{24031710}
Let $f: S \rightarrow$ L$(\R^m)$ be measurable. Then the inequality
\begin{equation*} 
\norm{U(f(s) \indikator{A}{s},s)} \le \norm{U(f(s),s)} \indikator{A}{s}+ 2 V(f(s),s).
\end{equation*}
holds for every $A \in \sigma(\mathcal{S})$ and $s \in S$.
\end{lemma}
\begin{proof}
With a little abuse of notation apply \eqref{eq:21031701} to $\tilde{R}:=\indikator{A}{s} I_m$ and $\tilde{x}:=f(s)x$. Then some simple calculations provide the desired conclusion.
\end{proof}
The previous Lemma can be regarded as a multivariate alternative for Lemma 2.8 in \cite{RajRo89}, whereas the following one uses some ideas from the proof of Theorem 3.2.2 in \cite{Ro87}.
\begin{lemma} \label{24031711}
For $f \in \mathcal{I}(M)$ let $(f_n)_{n \in \N}$ be a corresponding sequence of simple functions. Then for any $\eps_1,\eps_2>0$ there exists an $\zeta=\zeta(\eps_1,\eps_2)$ such that
\begin{equation*}
\forall n \ge \zeta \, \, \, \forall A \in \sigma(\mathcal{S}): \quad \Pro(\norm{I(f \indikatorzwei{A}) -I(f_n \indikatorzwei{A}) } \ge \eps_1) \le \eps_2.
\end{equation*}
\end{lemma}
\begin{proof}
Let $g_n:=f-f_n$. Then by linearity, \cref{27031702} and \cref{23111604} we have that 
\begin{equation} \label{eq:27031710}
\gamma_{g_n}(A):= \integral{A}{}{U( g_n(s),s) }{\lambda_M(ds)} \rightarrow 0, \quad A \in \sigma(\mathcal{S}).
\end{equation}
This convergence is even uniform in $A$. To prove this we define the measure
\begin{equation*}
\lambda_M^{*}(E):= \summe{l=1}{\infty} 2^{-l} \frac{\lambda_M(E \cap S_l)}{1+\lambda_M(S_l)}, \quad E \in \sigma(\mathcal{S}),
\end{equation*}
where $(S_l) \subset \mathcal{S}$ is a disjoint exhaustion of $S$ again. Then $A \mapsto \gamma_{g_n}(A)$ defines a vector measure with $\gamma_{g_n} \ll \lambda_M \ll \lambda_M^{*}$, i.e. the components $\gamma_{g_n}^{(k)}$ are signed measures with $\gamma_{g_n}^{(k)} \ll \lambda_M^{*}$ for every $n \in \N$ and $k=1,...,m$. Thus we can apply the Hahn-Saks-Vitali Theorem (see Proposition C.3 in \cite{Ryan}): For every $\eps>0$ there are $\delta_1,...,\delta_m>0$ fulfilling the implications
\begin{equation*}
\forall A \in \sigma(\mathcal{S}): \qquad \left(  \lambda_M^{*}(A) \le \delta_k  \quad \Rightarrow \quad \underset{n \in \N}{ \sup} \, |\gamma_{g_n}^{(k)}(A)| \le \eps \right)
\end{equation*}
for $k=1,...,m$. Hence there exists a $C>0$ such that the following assertion holds likewise with $\delta:=\min \{\delta_1,...,\delta_m\}$:
\begin{equation} \label{eq:27031705}
\forall A \in \sigma(\mathcal{S}): \qquad \left(  \lambda_M^{*}(A) \le \delta  \quad \Rightarrow \quad \underset{n \in \N}{ \sup} \, \norm{\gamma_{g_n}(A)} \le C   \,\eps \right).
\end{equation}
Using dominated convergence we have that $U_M(\cdot,s)$ is continuous for each $s \in S$ and therefore that $U_M(g_n(s),s) \rightarrow 0$ $\lambda_M$-a.e. Proceeding with Egorov's Theorem (note that $\lambda_M^{*}$ is finite) there exists a measurable set $D'$ such that the previous convergence is uniformly on $D'$ with $\lambda_M^{*}(S \setminus D') \le \delta/2$. Finally, we use $(S_l)$ and \cref{15031702} to verify that same is true on an appropriate set $D$ belonging to $\mathcal{S}$ with $\lambda_M^{*}(S \setminus D) \le \delta$. Especially we have $\lambda_M(D)< \infty$ as well as the following estimation for every $A \in \sigma(\mathcal{S})$:
\begin{align*}
\norm{\gamma_{g_n}(A)} & \le C   \,\eps  + \underset{s \in A \cap D}{\sup} \, \norm{U(g_n(s),s)} \cdot \lambda_M(A \cap D) \\
& \le C  \,\eps  + \underset{s \in  D}{\sup} \, \norm{U(g_n(s),s)} \cdot \lambda_M( D), 
\end{align*}
which obviously means that \eqref{eq:27031710} holds uniformly. Moreover, for $\R^m$-valued random vectors $X$ and $Y$, we can define $d(X,Y):= \int \min \{1,\norm{X-Y}\} \, d \Pro$ and know that $d$ is a metric whose induced convergence is equivalent to that in probability (when identifying random vectors which are equal a.s., see the proof of Theorem 6.7 in \cite{Kle08}). We now show for  $X_n(A):=I_M(g_n \indikatorzwei{A})-\gamma_{g_n}(A)$ that 
\begin{equation*}
c_n:= \underset{A \in \sigma(\mathcal{S})}{\sup} \, d(X_n(A),0) \in [0,2], \quad n \in \N
\end{equation*}
converges to zero. For this purpose choose $A_n \in \sigma(\mathcal{S})$ such that $c_n \le d(X_n(A_n),0)+1/n$. At the same time we have
\begin{equation*}
I_M(g_n)=X_n(A_n) + I_M(g_n \indikatorzwei{A_n^c})+ \gamma_{g_n}(A_n) =: X_n(A_n) + Y_n \rightarrow 0
\end{equation*}
in probability (see above). This also implies $X_n(A_n) \rightarrow 0$ by \cref{23111604} and monotonicity. For instance and provided that $X_n(A_n) \sim [0,Q_n,\phi_n]$ as well as $Y_n \sim [\tilde{\gamma}_n,\tilde{Q}_n,\tilde{\phi}_n]$ we obtain:
\begin{equation*}
0 \le \skp{Q_n t}{t} \le  \skp{Q_n t}{t} + \skp{\tilde{Q}_n t}{t} = \skp{(Q_n+\tilde{Q}_n)t}{t} \rightarrow 0, \quad t \in \R^m
\end{equation*}
since $Q_n+\tilde{Q}_n$ equals the Gaussian component of $I_M(g_n)$ by independence of $X_n(A_n)$ and $Y_n$ (see \cref{21081701} (d) and Proposition 3.1.21 in \cite{thebook}). Hence $c_n \rightarrow 0$. Furthermore, we see that $d(I_M(g_n \indikatorzwei{A}),0) \le d(X_n(A),0)+\norm{\gamma_{g_n}(A)}$ holds for every $A \in \sigma(\mathcal{S})$ and $n \in \N$ due to the fact that $[0,\infty) \ni x \mapsto \min \{1,x\}$ is subadditive. By what we have seen before this shows that $d(I_M(g_n \indikatorzwei{A}),0)$ converges to $0$ uniformly in $A \in \sigma(\mathcal{S})$. Finally let $0< \eps_1 \le 1$ arbitrary ($\eps_1>1$ obvious), then we obtain the assertion by reading this convergence together with
\begin{equation*}
\Pro(\norm{ I(f \indikatorzwei{A}) -I(f_n \indikatorzwei{A}) } \ge \eps_1)  =  \Pro(\norm{I(g_n \indikatorzwei{A}) } \ge \eps_1) 
 \le \eps_1^{-1} \underset{A \in \sigma(\mathcal{S})}{\sup} \, d(I(g_n \indikatorzwei{A}) ,0),
\end{equation*}
where we used that $\Pro(\norm{X} \ge \eps_1) \le d(X,0) / \eps_1$ (for any random vector $X$).
\end{proof}
\begin{theorem} \label{22031701}
Let $f: S \rightarrow$ L$(\R^m)$ be $\sigma(\mathcal{S})$-$\B(\Li{m})$-measurable. Then the following statements are equivalent:
\begin{itemize}
\item[(I)] $f \in \mathcal{I}(M)$.
\item[(II)] The integrals $\gamma_f$ as well as $Q_f$ exist and $\phi_f$ is a L\'{evy} measure.
\item[(III)] The integral in \eqref{eq:21031707} exists for every $t \in \R^m$ and the mapping in \eqref{eq:13071701} is continuous.
\end{itemize}
\end{theorem}
\begin{proof}
In view of what we pointed out before, especially in the proof of \cref{27031702}, it obviously suffices to show that (II) implies (I). Throughout the proof let $(S_n') \subset \mathcal{S}$ be an increasing sequence whose union is $S$ and write $f(s)=(f^{i,j}(s))_{i,j=1,...,m}$ for every $s \in S$. \\
\underline{First step:} We define $S_n:=S_n' \cap \{s:|f^{i,j}(s)|<n \text{ for all $1 \le i,j \le m$ }\} \in \mathcal{S}$ with $S_n \uparrow S$ and thereafter the sequence $(f_n)$ of $\mathcal{S}$-simple functions (see \cref{15031702}) via
\begin{equation*}
f_n^{i,j}(s):= \indikator{S_n}{s} \cdot \begin{cases} \frac{l}{n}, & \text{if $\frac{l}{n} \le f^{i,j}(s) < \frac{l+1}{n}$ for $l=0,...,n^2-1$} \\  -\frac{l}{n}, & \text{if $- \frac{l+1}{n} < f^{i,j}(s) \le - \frac{l}{n}$ for $l=0,...,n^2-1$} \\
0, & \text{if $|f^{i,j}(s)| \ge n$.}\end{cases} 
\end{equation*}
Hence we see that $f_n \rightarrow f$ pointwise with $|f_n^{i,j}(s)| \le |f^{i,j}(s)|$ for every $s \in S$, whereas $|f_n^{i,j}(s)-f^{i,j}(s)| \le 1/n$ merely holds for $s \in S_n$. Moreover, there exist $C_1,C_2>0$ such that $\norm{f_n(s)} \le C_1 \norm{f(s)}$ for all $s \in S$ and $\norm{f_n(s)-f(s)}$ is bounded by $C_2/n$ as long as $s \in S_n$. Particularly we obtain for all $j \ge n$ and $s \in S$: 
\begin{align} 
\norm{f_n(s) -f_{j}(s)} & \le C_1 \, \norm{f(s)} \, \indikator{S_{j} \setminus S_n}{s} + 2C_2  \, \indikator{ S_{n} }{s}  . \label{eq:23031701}
\end{align}
\underline{Second step:} Next we show that $g^{(k)}:= f \indikatorzwei{S_k} \in \mathcal{I}(M)$ for $k \in \N$ arbitrary by means of the $\mathcal{S}$-simple sequence $(g_n^{(k)})_n$ which is defined via $g_n^{(k)}:= f_n \indikatorzwei{S_k}$. Obviously, we have $g_n^{(k)} \rightarrow g^{(k)}$ pointwise and with $C:=2C_1$ one confirms that
\begin{equation} \label{eq:23031702}
\norm{g_n^{(k)}(s)-g_{j}^{(k)}(s)} \le C\, \indikator{ S_{k}}{s}
\end{equation}
is true for all $j \ge n \ge k$ and $s \in S$ due to \eqref{eq:23031701}. In view of \cref{20031706} it suffices to show that $(I_M(g_n^{(k)} \indikatorzwei{A}))_n$ converges in probability. For this purpose we now fix an arbitrary sequence $n_1<j_1<n_2<...$ of increasing natural numbers and prove that the convergences
\allowdisplaybreaks
\begin{align}
\integral{S}{}{U_M \left ((g_{n_l}^{(k)}(s)-g_{j_l}^{(k)}(s)) \indikator{A}{s},s \right)}{\lambda_M(ds)} \rightarrow 0,  \label{eq:24031701} \\
\integral{A}{}{ (g_{n_l}^{(k)}(s)-g_{j_l}^{(k)}(s)) \, \beta_M(s) \, (g_{n_l}^{(k)}(s)-g_{j_l}^{(k)}(s))^{*} }{\lambda_M(ds)} \rightarrow 0,   \label{eq:24031702}  \\
\integral{S}{}{V_M \left ((g_{n_l}^{(k)}(s)-g_{j_l}^{(k)}(s)) \indikator{A}{s},s \right)}{\lambda_M(ds)} \rightarrow 0 \label{eq:24031703} \,
\end{align}
hold for $l \rightarrow \infty$, respectively. By continuity of $U_M(\cdot,s)$ and $V_M(\cdot,s)$ it is first clear that the integrands in \eqref{eq:24031701}-\eqref{eq:24031703} converge to zero for every $s \in S$. Then the assertion follows by dominated convergence in each case: For \eqref{eq:24031702} use \eqref{eq:23031702} and observe that $\norm{\beta_M(s)} \indikator{A \cap S_k}{s}$ is $\lambda_M$-integrable. On the other hand we see that the integrand in \eqref{eq:24031703} is dominated by $V_M(C \indikator{A \cap S_k}{s}I_m,s)$ (here and below at least for $l$ sufficiently large), whereas \eqref{eq:06031702} and \cref{08031703} provide the following steps that have been performed similarly before:
\allowdisplaybreaks
\begin{align*}
 \integral{S}{}{ V_M(C \indikator{A \cap S_k}{s}I_m,s)}{\lambda_M(ds)} & \le   (1+C^2) \integral{S \times \R^m}{}{  \min\{1, \norm{x }^2  \} \indikator{A \cap S_k}{s}}{(\lambda_M \odot \rho_M)(ds,dx)} \\
& =  (1+C^2) \integral{\R^m}{}{ \min\{1, \norm{x }^2  \} }{\phi_{A \cap S_k}(dx)} \\
& <  \infty.
\end{align*}
Using \eqref{eq:21031701} we can finally argue likewise that the integrand in \eqref{eq:24031701} is dominated by 
\begin{equation*}
s \mapsto \left(C \norm{\alpha_m(s)} +C' \integral{\R^m}{}{\min \{1,\norm{x}^2\}}{\rho_M(s,dx)}\right) \indikator{A \cap S_k}{s}
\end{equation*}
with $C':=\max \{2,C+C^3\}$ as well as that the mapping we mentioned recently is $\lambda_M$-integrable. Finally suppose that $(I_M(g_n^{(k)} \indikatorzwei{A}))_n$ would not converge in probability, then it would not be Cauchy either (in view and in the sense of Corollary 6.15 in \cite{Kle08}). Hence we obtain a sequence $n_1<j_1<n_2<...$ as above sucht that $I_M(g_{n_l}^{(k)} \indikatorzwei{A}) - I_M(g_{j_l}^{(k)} \indikatorzwei{A})=I_M((g_{n_l}^{(k)}-g_{n_l}^{(k)} ) \indikatorzwei{A})$ neither converges in probability to zero nor in distribution. By \cref{27031702} and in view of \cref{23111604} together with  \eqref{eq:24031701}-\eqref{eq:24031703} this gives the contradiction.\\
\underline{Third step:} For $A \in \sigma(\mathcal{S})$ arbitrary we further conclude that there is an increasing sequence $(j_l^{A})$ of natural numbers which fullfils the following implication for every $l \in \N$:
\begin{equation} \label{eq:24031715}
 k_1,k_2 \ge j^A_l \quad \Rightarrow  \quad \Pro \left(\norm{I(g^{(k_1)} \indikatorzwei{A} ) - I(g^{(k_2)} \indikatorzwei{A} )} \ge 1/ l  \right) \le 1 / l.
\end{equation}
Similar to the previous step this is again equivalent to the following assertions
\allowdisplaybreaks
\begin{align}
\integral{S}{}{U_M \left ((g^{(l_k)}(s)-g^{(n_k)}(s)) \indikator{A}{s},s \right)}{\lambda_M(ds)} \rightarrow 0,  \label{eq:24031721} \\
\integral{A}{}{ \left (g^{(l_k)}(s)-g^{(n_k)}(s) \right) \, \beta_M(s) \, \left (g^{(l_k)}(s)-g^{(n_k)}(s) \right)^{*} }{\lambda_M(ds)} \rightarrow 0,   \label{eq:24031722}  \\
\integral{S}{}{V_M \left ((g^{(l_k)}(s)-g^{(n_k)}(s)) \indikator{A}{s},s \right)}{\lambda_M(ds)} \rightarrow 0  \label{eq:24031723} 
\end{align}
for $k  \rightarrow \infty$, respectively and with any fixed sequence $n_1<l_1<n_2<...$ as before. In virtue of $(S_{l_k} \setminus S_{n_k}) \subset (S \setminus S_k) \downarrow \emptyset$ we only have to find $\lambda_M$-integrable functions again which dominate the previous integrands. Concerning \eqref{eq:24031722} and \eqref{eq:24031723} this is obvious as we assume the existence of $Q_f$ and the $\lambda_M$-integrability of $V_M(f(\cdot),\cdot)$. For \eqref{eq:24031721} we use \cref{24031710} and then again the assumption on $V_M(f(\cdot),\cdot)$ as well as the one on $U_M(f(\cdot),\cdot)$.\\
\underline{Fourth step:} Inductively \cref{24031711} provides a sequence $(\zeta_k)$ of increasing natural numbers such that
\begin{equation} \label{eq:27031701}
\forall A \in \sigma(\mathcal{S}) \, \, \, \forall k \in \N: \qquad \Pro \left(\norm{I(g^{(k)} \indikatorzwei{A} ) -I(g_{\zeta_k}^{(k)} \indikatorzwei{A} ) } \ge 1/k  \right) \le 1/k.
\end{equation}
Then we replace the sequence $(f_k)$ from the first step by $f_k:=g^{(k)}_{\zeta_k}$ and realize that $f_k \rightarrow f$ pointwise again. Let $A \in \sigma(\mathcal{S})$ as well as $\eps_1,\eps_2>0$ be arbitrary. Then the following calculation yields that $(I_M(f_k \indikatorzwei{A}))$ is a Cauchy sequence w.r.t. convergence in probability. In fact we choose a $K_0 \in \N$ such that $K_0^{-1} \le \min \{\eps_1,\eps_2\} / 3$ and set $K:= \max \{K_0,j_{K_0}^A\}$. Then for any $k_1, k_2 \ge K$ we we get using \eqref{eq:24031715} and \eqref{eq:27031701} that 
\allowdisplaybreaks
\begin{align*}
& \Pro \left( \norm{I(f_{k_1} \indikatorzwei{A}) - I(f_{k_2} \indikatorzwei{A})} \ge \eps_1 \right) \\
& \le \Pro \left( \norm{I( g^{(k_1)}_{\zeta_{k_1}}  \indikatorzwei{A}) - I(g^{(k_1)} \indikatorzwei{A})} \ge K_0^{-1} \right) + \Pro \left( \norm{I( g^{(k_1)} \indikatorzwei{A}) - I(g^{(k_2)} \indikatorzwei{A})} \ge K_0^{-1} \right)   \\
&  \qquad \qquad +\Pro \left( \norm{I( g^{(k_2)} \indikatorzwei{A}) - I(g^{(k_2)}_{\zeta_{k_2}} \indikatorzwei{A})}  \ge K_0^{-1} \right)  \\
& \le \Pro \left( \norm{I( g^{(k_1)}_{\zeta_{k_1}}  \indikatorzwei{A}) - I(g^{(k_1)} \indikatorzwei{A})} \ge k_1^{-1} \right) + \Pro \left( \norm{I( g^{(k_1)} \indikatorzwei{A}) - I(g^{(k_2)} \indikatorzwei{A})} \ge K_0^{-1} \right) \\
& \qquad \qquad  + \Pro \left( \norm{I( g^{(k_2)} \indikatorzwei{A}) - I(g^{(k_2)}_{\zeta_{k_2}} \indikatorzwei{A})}  \ge k_2^{-1} \right) \\
& \le  k_1^{-1} + K_0^{-1}  +k_2^{-1}  \\
& \le \eps_2
\end{align*}
and the proof is complete.
\end{proof}
With $f_j=\indikatorzwei{A_j} I_m$ and the following result, which extends the conclusion in \cite{simulation}, we see that the infinite divisibility of an ISRM implicitly extends to its finite dimensional distributions. 
\begin{cor} \label{20041701}
For $f_1,...,f_n \in \mathcal{I}(M)$ the random vector $(I_M(f_1),...,I_M(f_n))$ has an i.d. distribution.
\end{cor}
\begin{proof}
Denote the characteristic function of $(I_M(f_1),...,I_M(f_n))$ by $\varphi$ and fix some arbitrary $l \in \N$. Then it suffices to show that the function $\varphi^{1/l}$, which we should not understand in any logarithmic sense (see \cref{21081701} (b) instead), also describes a characteristic function on $\R^{n \cdot m}$. Thus if $M(A) \sim [\gamma_A,Q_A ,\phi_A]$, we see that $M'$ with $M' (A) \sim [l^{-1} \gamma_A,l^{-1}  Q_A, l^{-1}  \phi_A]$ (for every $A \in \mathcal{S}$) is also a valid ISRM according to \cref{09121603}. Then \cref{22031701} leads to $\mathcal{I}(M)=\mathcal{I}(M')$ such that $(I_{M'}(f_1),...,I_{M'}(f_n))$ has the characteristic function $\varphi^{1/l}$. 
\end{proof} 
For the rest of this paper we briefly want to study the close relation between $\mathbb{K}=\R$ and $\mathbb{K}=\C$ which can be clarified by introducing the \textit{(partially) associated mapping of $f$}, namely $\tilde{f},\tilde{f}_{p}:S \rightarrow$ L$(\R^{2m})$ by
\begin{equation*}
\tilde{f}(s):=  \begin{pmatrix} \text{Re } f(s) & -\text{Im } f(s) \\ \text{Im } f(s) & \text{Re } f(s) \end{pmatrix}  \qquad  \text{and} \qquad \tilde{f}_{p}(s):=  \begin{pmatrix} \text{Re } f(s) & -\text{Im } f(s) \\ 0 & 0 \end{pmatrix},
\end{equation*}
where $f:S \rightarrow$ L$(\C^m)$ is arbitrary. More precisely and with regard to \cref{19041701} we get the following observation in which we assume $M$ to be a $\C^m$valued i.d. ISRM. %
\begin{prop} \label{20031702}
For $f:S \rightarrow L(\mathbb{C}^m)$ we have: $f$ is $M$-integrable if and only if $\tilde{f}$ is $\Xi(M)$-integrable and in this case $\Xi(I_M(f \indikatorzwei{A}))=I_{\Xi(M)}(\tilde{f} \indikatorzwei{A})$ a.s for every $A \in \sigma(\mathcal{S})$. Similarly $f$ is partially $M$-integrable if and only if $\tilde{f}_p$ is $\Xi(M)$-integrable and in this case $\Xi(\text{Re }I_M(f \indikatorzwei{A}))=I_{\Xi(M)}(\tilde{f}_p \indikatorzwei{A})$ a.s. for every $A \in \sigma(\mathcal{S})$.
\end{prop}
\begin{proof}
This follows by a simple calculation using \eqref{eq:15031701} and passing through the limit.
\end{proof}
On one hand this immediately allows us to apply \cref{22031701} and \cref{20041701} accordingly. On the other hand it shows that the complex-valued perspective mostly simplifies the description of several problems that actually have a real origin. We derive the following.
\begin{cor} \label{20031701}
Let $M$ be as before, particularly $\C^m$-valued.
\begin{itemize}
\item[(a)] If $f \in \mathcal{I}(M)$, then $I_M(f \indikatorzwei{A})$ is well-defined and i.d. for every $A \in \sigma(\mathcal{S})$, whereas the log-characteristic function of $\Xi(I_M(f \indikatorzwei{A}))$ is given by
\begin{equation} \label{eq:20031710}
\R^{2m} \ni t \mapsto    \integral{A}{}{ K_{\Xi(M)} ( \tilde{f}(s)^{*}t ,s)}{\lambda_{\Xi(M)}(ds)} =   \integral{A}{}{ K_{\Xi(M)} ( \Xi(f(s)^{*} z)  ,s)}{\lambda_{\Xi(M)}(ds)}
\end{equation}
with $z:=\Xi^{-1}(t) \in \C^m$.
\item[(b)] If $f_1,...,f_n \in \mathcal{I}(M)$, then we have for any $t_1,...,t_n \in \R^{2m}$:
\begin{equation*}
\bigerwartung{  \text{e}^{\im \summezwei{j=1}{n} \skp{\Xi(I(f_j))}{t_j}  }} = \exp \left( \, \integral{S}{}{K_{\Xi(M)} \left( \summezwei{j=1}{n} \tilde{f_j}(s)^{*} t_j ,s \right)}{\lambda_{\Xi(M)}(ds)}   \right).
\end{equation*}
\item[(c)] For $f_,f_1,f_2,... \in \mathcal{I}(M)$ we have that $I_M(f_n) \rightarrow I_M(f)$ in probability is equivalent to 
\begin{equation*} 
\integral{\R^m}{}{K_{\Xi(M)}(  (\tilde{f_n}(s)-\tilde{f}(s))^{*} t ,s)}{\lambda_{\Xi(M)}(ds)} \rightarrow 0, \quad t \in \R^{2m}.
\end{equation*}
\item[(d)] Let $f_1,f_2 \in \mathcal{I}(M)$ such that $\norm{\tilde{f_1}(s)} \cdot \norm{\tilde{f_2}(s)}=0$ holds $\lambda_{\Xi(M)}$-a.e. Then $I_M(f_1)$ and $I_M(f_2)$ are independent.
\end{itemize}
\end{cor}
\begin{proof}
In view of \cref{20031702} part (a) follows by \cref{21081701} and the claimed equality can be checked immediately. And since, by linearity, $I_M(f_n) \rightarrow I_M(f)$ is equivalent to $\Xi(I_M(f_n-f)) \rightarrow 0$ in probability, this gives (c) again. Moreover, \cref{20031702} says that the assertion in (d) is equivalent to the independence of $I_{\Xi(M)}(\tilde{f_1})$ and $I_{\Xi(M)}(\tilde{f_2})$ such that the proof reduces to the case $\mathbb{K}=\R$. Finally we write $t_j=(t_{j,1},t_{j,2})$ as well as $t_{j,i}= Q_{j,i} e$ with $e=(1,..,1) \in \R^m$ and $Q_{j,i} \in \Li{m}$ suitable. Then for $R_j:= \frac{1}{2} (R_{j,1}+R_{j,2}), Q_j:= \frac{1}{2} (R_{j,1}-R_{j,2}) $ and $V_j:=R_j - \im Q_j \in$ L$(\C^m)$ we observe similar to Proposition 6.2.1 in \cite{SaTaq94} that
\begin{align*}
& \summe{j=1}{n} \skp{\Xi(I_M(f_j))}{t_j} \\
& \summe{j=1}{n} \skp{R_{j,1}^{*} (\text{Re } I_M(f_j)) + R_{j,2}^{*} (\text{Im } I_M(f_j)) }{e}  \\
 &= \summe{j=1}{n} \skp{R_j^{*} (\text{Re } I_M(f_j)) -Q_j^{*} (\text{Im } I_M(f_j)) +Q_j^{*} (\text{Re } I_M(f_j))+ R_j^{*} (\text{Im } I_M(f_j))  }{e} \\
 &= \summe{j=1}{n} \skp{\text{Re } V_j^{*}I_M(f_j) + \text{Im } V_j^{*} I_M(f_j)}{e} \\
 &= \skpzwei{\Xi \left ( I_M  \left (\sum_{j=1}^{n} V_j^{*} \cdot f_j \right) \right )}{\begin{pmatrix} e \\ e \end{pmatrix}}
\end{align*}
by both parts of \cref{04051701}. Verify the identity
\begin{equation}  \label{eq:11051701}
\left(\summe{j=1}{n} V_j^{*} f_j(s) \right)^{*} (e + \im e)=\summe{j=1}{n} f_j(s)^{*} (t_{j,1}+\im t_{j,2})=\summe{j=1}{n} \tilde{f_j}(s)^{*}t_j , \quad s \in S
\end{equation}
to see that (b) follows by (a). 
\end{proof}
\begin{remark} 
We also observe that $\Xi (f(s)^{*}t_1 )$ equals $\tilde{f}_p(s)^{*} t$ for every $t=(t_1,t_2) \in \R^{2m}$. Then the properties for the partial case (see \cref{20031706}) can be formulated and proved similarly which is therefore left to the reader. We merely note that the following key relation holds for any $f_1,...,f_n \in \mathcal{I}_p(M)$ and $t_1,...,t_n \in \R^{m}$.
\begin{equation} \label{eq:21081710}
\bigerwartung{  \text{e}^{\im \summezwei{j=1}{n} \skp{\text{Re }I(f_j)}{t_j}  }} = \exp \left( \, \integral{S}{}{K_{\Xi(M)} \left( \Xi \left( \summezwei{j=1}{n} f_j(s)^{*} t_j \right),s \right)}{\lambda_{\Xi(M)}(ds)}   \right).
\end{equation}
\end{remark}
\section*{Acknowledgement}
The results of this paper are part of the first author's PhD-thesis (see \cite{Diss}), written under the supervision of the second named author.
%% References


\begin{thebibliography}{10}

\bibitem{multi}
H.~Bierm{\'e}, C.~Lacaux, and H.-P. Scheffler.
\newblock Multi-operator scaling random fields.
\newblock {\em Stochastic Processes and their Applications},
  121(11):2642--2677, 2011.

\bibitem{bms}
H.~Bierm{\'e}, M.~M. Meerschaert, and H.-P. Scheffler.
\newblock Operator scaling stable random fields.
\newblock {\em Stochastic Processes and their Applications}, 117(3):312--332,
  2007.

\bibitem{Dud02}
R.~M. Dudley.
\newblock {\em Real analysis and probability}, volume~74.
\newblock Cambridge University Press, 2002.

\bibitem{DunSch57}
N.~Dunford, J.~T. Schwartz, W.~G. Bade, and R.~G. Bartle.
\newblock {\em Linear operators}.
\newblock Wiley-interscience New York, 1971.

\bibitem{falconer3}
K.~Falconer and L.~Liu.
\newblock Multistable processes and localizability.
\newblock {\em Stochastic Models}, 28(3):503--526, 2012.

\bibitem{hoffmann}
A.~Hoffmann.
\newblock {\em Operator Scaling Stable Random Sheets with application to binary
  mixtures}.
\newblock PhD thesis, University of Siegen, 2011.

\bibitem{Jac01}
N.~Jakob.
\newblock {\em Pseudo Differential Operators and Markov Processes. Vol. 1}.
\newblock Imperial College Press, London, 2001.

\bibitem{simulation}
W.~Karcher, H.-P. Scheffler, and E.~Spodarev.
\newblock Infinite divisibility of random fields admitting an integral
  representation with an infinitely divisible integrator.
\newblock {\em arXiv preprint arXiv:0910.1523}, 2009.

\bibitem{Kle08}
A.~Klenke.
\newblock {\em Wahrscheinlichkeitstheorie}, volume~1.
\newblock Springer, 2006.

\bibitem{Diss}
D.~Kremer.
\newblock {\em Multivariate stochastische Integrale mit Anwendung am Beispiel
  Operator-stabiler und Operator-selbst\"ahnlicher Zufallsfelder}.
\newblock PhD thesis, University of Siegen, 2017.

\bibitem{paper2}
D.~Kremer and H.-P. Scheffler.
\newblock Operator-stable and operator-self-similar random fields.
\newblock {\em In preperation}, 2017.

\bibitem{La68}
S.~Lang.
\newblock {\em Real Analysis}.
\newblock Addison-Wesley, 1983.

\bibitem{lixiao}
Y.~Li and Y.~Xiao.
\newblock Multivariate operator-self-similar random fields.
\newblock {\em Stochastic Processes and their Applications}, 121(6):1178--1200,
  2011.

\bibitem{thebook}
M.~M. Meerschaert and H.-P. Scheffler.
\newblock {\em Limit distributions for sums of independent random vectors:
  Heavy tails in theory and practice}, volume 321.
\newblock John Wiley \& Sons, 2001.

\bibitem{Pre56-2}
A.~Pr{\'e}kopa.
\newblock Extension of multiplicative set functions with values in a banach
  algebra.
\newblock {\em Acta Mathematica Hungarica}, 7(2):201--213, 1956.

\bibitem{Pre56}
A.~Pr{\'e}kopa.
\newblock On stochastic set functions. i.
\newblock {\em Acta Mathematica Hungarica}, 7(2):215--263, 1956.

\bibitem{Pre57}
A.~Pr{\'e}kopa.
\newblock On stochastic set functions. iii.
\newblock {\em Acta Mathematica Hungarica}, 8(3-4):375--400, 1957.

\bibitem{RajRo89}
B.~S. Rajput and J.~Rosinski.
\newblock Spectral representations of infinitely divisible processes.
\newblock {\em Probability Theory and Related Fields}, 82(3):451--487, 1989.

\bibitem{Ro87}
J.~Rosi{\'n}ski.
\newblock {\em Bilinear random integrals}.
\newblock Instytut Matematyczny Polskiej Akademi Nauk Warszawa, 1987.

\bibitem{Ryan}
R.~A. Ryan.
\newblock {\em Introduction to tensor products of Banach spaces}.
\newblock Springer Science \& Business Media, 2013.

\bibitem{SaTaq94}
G.~Samoradnitsky and M.~S. Taqqu.
\newblock {\em Stable non-Gaussian random processes: stochastic models with
  infinite variance}, volume~1.
\newblock CRC press, 1994.

\bibitem{Sato}
K.-I. Sato.
\newblock {\em L{\'e}vy processes and infinitely divisible distributions}.
\newblock Cambridge university press, 1999.

\bibitem{Woy70}
W.~A. Woyczy{\'n}ski.
\newblock {\em Ind-additive functionals on random vectors}.
\newblock Instytut Matematyczny Polskiej Akademi Nauk Warszawa, 1970.

\end{thebibliography}
\end{document}